\documentclass[12pt, reqno]{amsart}
\usepackage{mathrsfs}
\usepackage{amssymb,amsthm,amsmath}
\usepackage[numbers,sort&compress]{natbib}
\usepackage{amssymb,amsmath}
\usepackage{amsfonts}
\usepackage{mathrsfs}
\usepackage{latexsym}
\usepackage{amssymb}
\usepackage{amsthm}
\usepackage{color}
\usepackage{pdfsync}
\usepackage{indentfirst}
\date{today}

\usepackage{hyperref}
\hypersetup{hypertex=true,
	colorlinks=true,
	linkcolor=blue,
	anchorcolor=g,
	citecolor=red}

\usepackage{amsmath}
\usepackage{amsthm}

\newtheorem{remark}{Remark}[section]

\newtheorem{theorem}{Theorem}[section]
\newtheorem{proposition}{Proposition}[section]
\newtheorem{lemma}{Lemma}[section]

\newcommand{\beq}{\begin{equation}}
	\newcommand{\eeq}{\end{equation}}
\newcommand{\ben}{\begin{eqnarray}}
	\newcommand{\een}{\end{eqnarray}}
\newcommand{\beno}{\begin{eqnarray*}}
	\newcommand{\eeno}{\end{eqnarray*}}

\numberwithin{equation}{section}

\begin{document}
	\title[Suppression of blow-up via the 2-D Couette flow]{Stability of a class of supercritical volume-filling chemotaxis-fluid model near Couette flow}
	\author{Lili~Wang}
	\address[Lili~Wang]{School of Mathematical Sciences, Dalian University of Technology, Dalian, 116024,  China}
	\email{wanglili\_@mail.dlut.edu.cn}
	\author{Wendong~Wang}
	\address[Wendong~Wang]{School of Mathematical Sciences, Dalian University of Technology, Dalian, 116024,  China}
	\email{wendong@dlut.edu.cn}
	\author{Yi~Zhang}
	\address[Yi~Zhang]{School of Mathematical Sciences, Dalian University of Technology, Dalian, 116024,  China}
	\email{zysx@mail.dlut.edu.cn}
	\date{\today}
	\maketitle
	

	\begin{abstract} 
		Consider a class of chemotaxis-fluid model incorporating a volume-filling effect in the sense of Painter
and Hillen (Can. Appl. Math. Q. 2002; 10(4): 501-543), which is a supercritical parabolic-elliptic Keller-Segel system.  As shown by Winkler et al., for any given mass, there exists a corresponding solution of the same mass
that blows up in either finite or infinite time. In this paper, we investigate  the stability properties of the two dimensional Patlak-Keller-Segel-type chemotaxis-fluid model near the Couette flow $ (Ay, 0) $ in $ \mathbb{T}\times\mathbb{R}, $ and show that the solutions are global in time as long as the initial cell mass $M<\frac{2\pi}{\sqrt{3}} $ and the  shear flow is sufficiently strong ($A$ is large enough).

	\end{abstract}
	
	{\small {\bf Keywords:} Keller-Segel; Chemotaxis-fluid model;
		Couette flow;
		Enhanced dissipation; Suppression;
		Blow-up}
	
	\section{Introduction}

Consider the following two-dimensional parabolic-elliptic chemotaxis-fluid model in $ \mathbb{T}\times\mathbb{R} $ with $ \mathbb{T}=[0,2\pi) $:
	\begin{equation}\label{ini}
		\left\{
		\begin{array}{lr}
			\partial_tn+u\cdot\nabla n=\nabla\cdot(\phi(n)\nabla n)-\nabla\cdot(\psi(n)\nabla c), \\
			\triangle c+n-c=0, \\
			\partial_tu+ u\cdot\nabla u+\nabla P=\triangle u+n\nabla\Phi_0,\quad\nabla\cdot u=0, \\
			(n,u)\big|_{t=0}=(n_{\rm in},u_{\rm in}),
		\end{array}
		\right.
	\end{equation}
where the function $n$ represents the cell density, $c$ denotes the chemoattractant density, and $u$ denotes the velocity of fluid. In addition, $P$ is the pressure and $\Phi_0$ is the potential function. The functions $\phi$ and $\psi$ are assumed to belong to $C^2([0,\infty))$ and to satisfy $\phi>0$ on $[0,\infty)$, $\psi(0)=0$ and $\psi> 0$ on $(0,\infty)$.
Based on a biased random
walk analysis, Hillen and Painter \cite{painter}  derived a functional link between the self-diffusivity $\phi(n)$ and the chemotactic sensitivity $\psi(n)$ by assuming that the cells’ movement is inhibited near points where the cells
are densely packed.  That, in a
non-dimensional version, takes the form
\beno
\phi(n)=Q(n)-uQ'(n),\quad
\psi(n)=nQ(n),
\eeno
where $Q(n)$ denotes the density-dependent probability for a cell to find space somewhere in its current neighborhood. Since this
probability is basically unknown, different choices for $ Q $ are conceivable, each of these providing a certain version of (\ref{ini}) that
incorporates this so-called volume-filling effect (see \cite{winkler0}). Here we choose $Q(n)=n+1.$
Assume $\Phi_0=y$ as Zeng-Zhang-Zi in \cite{zeng}. Then \eqref{ini} becomes 
	\begin{equation}\label{ini-}
		\left\{
		\begin{array}{lr}
			\partial_tn+u\cdot\nabla n=\triangle n-\nabla\cdot((n^2+n)\nabla c), \\
			\triangle c+n-c=0, \\
			\partial_tu+ u\cdot\nabla u+\nabla P=\triangle u+ne_2,\quad\nabla\cdot u=0, \\
			(n,u)\big|_{t=0}=(n_{\rm in},u_{\rm in}).
		\end{array}
		\right.
	\end{equation}


	Let us  briefly recall some developments on the system (\ref{ini}).  If $ u=0 $ and $ \Phi_0=0 $, the system (\ref{ini}) is reduced to the classical parabolic-elliptic Patlak-Keller-Segel system, which is proposed as a macroscopic model for chemotactic cell migration and was jointly developed by Patlak \cite{Patlak1}, Keller and Segel \cite{Keller1}. In details, the classical Patlak-Keller-Segel (PKS) system is stated as follows:
	\begin{equation}\label{eq:pks}
		\left\{
		\begin{array}{lr}
			\partial_tn=\nabla\cdot(\phi(n)\nabla n)-\nabla\cdot(\psi(n)\nabla c), \\
			\tau\partial_{t}c=\triangle c+n-c,
		\end{array}
		\right.
	\end{equation}
	where $ \tau=1, 0 $ corresponds to the parabolic-parabolic case and the parabolic-elliptic case, respectively.
	For $\phi(n)=1$ and $ \psi(n)=n, $ (\ref{eq:pks}) is the most common formulation of the Patlak-Keller-Segel model, and the last term of $(\ref{eq:pks})_1$ has a critical Fujita exponent in 2D. 
	As long as the dimension of space is higher than one, the solutions of the classical PKS system may blow up in finite time. 
	The 2D PKS model of parabolic-parabolic has a critical mass of $8\pi$, if the cell mass $M:=||n_{\rm in}||_{L^1}$ is less than $8\pi$,  Calvez-Corrias \cite{Calvez1} obtained the solutions of the system are global in time,
	while the cell mass is greater than $8\pi$, the solutions will blow up in finite time by Schweyer in \cite{Schweyer1}.
	Moreover, the 2D parabolic-elliptic Patlak-Keller-Segel system is globally well-posed if and only if the total mass $M\leq8\pi$ by Wei in \cite{wei11}.
	When the spatial dimension is higher than two, the solutions of the PKS system will blow up for any initial mass, meaning that
	no mass threshold for aggregation exists in that case (for example, see  \cite{nagai1995}, \cite{Na},  \cite{sw2019}, \cite{winkler1} etc.).
Generally, for the supercritical case of  $ \psi(n)=n(n+1)^{\alpha} $ with $\alpha>0$, blow-up phenomena may occur for any initial mass in 2D.
For example, Winkler in \cite{winkler0} mentioned a blow-up criterion: 
\ben
\frac{\psi(n)}{\phi(n)}\geq C_0 n
\een
for some $C_0>0$ and sufficiently large $n$ in the two-dimensional case, though he considered the parabolic-parabolic case at that time.
 When the spatial dimension is equal to two or greater than or equal to 3,  Horstmann-Winkler \cite{HW} showed that there exist solutions of (\ref{eq:pks})  that become unbounded for $ \alpha>0 $. For more related results, we refer to \cite{SM,wang2024,winkler0,winkler2} and the references therein.
	
	
	{\it An interesting question:  Can the stabilizing effect of the moving fluid suppress the finite-time blow-up?} 
	
	As we all know, when considering $ \phi(n)=1 $ and $ \psi(n)=n, $ some progresses have been made in proving the suppression of the chemotactic blow-up by the presence of fluid flow as follows:
	
	{\bf 2D case.} For the parabolic-elliptic PKS system in 2D, Kiselev-Xu \cite{Kiselev1} suppressed the  blow-up by stationary relaxation
	enhancing flows or time-dependent Yao-Zlatos near-optimal mixing flows in $\mathbb{T}^d$.
He \cite{he0} investigated the suppression of blow-up by a large strictly monotone shear flow for the parabolic-parabolic PKS model in $\mathbb{T}\times\mathbb{R}$. 
	For the coupled PKS-NS system, Zeng-Zhang-Zi firstly considered the 2D PKS-NS system near the Couette flow in $\mathbb{T}\times\mathbb{R}$ \cite{zeng}.
	Li-Xiang-Xu \cite{Li0} studied the suppression of blow-up in PKS-NS system via the Poiseuille flow in $\mathbb{T}\times\mathbb{R}$, and Cui-Wang considered Poiseuille flow with the Navier-slip boundary of PKS-NS system and obtained the solutions are global without any smallness condition \cite{cui1}.
	
	{\bf 3D case.} For the PKS system of parabolic-elliptic case in 3D, there are few research results in this area. Bedrossian-He investigated the suppression of blow-up by shear flows in $\mathbb{T}^3$ and $\mathbb{T}\times\mathbb{R}^2$ by assuming the initial mass is less than $8\pi$ in \cite{Bedro2}. 
	Shi-Wang \cite{SW} considered the suppression effect of the flow $(y,y^2,0)$ in $\mathbb{T}^2\times\mathbb{R}$, and Deng-Shi-Wang \cite{wangweike1} proved the Couette flow with a sufficiently large amplitude prevents
	the blow-up of solutions in the whole space. Besides, for the stability effect of buoyancy,
	Hu-Kiselev-Yao considered the blow-up suppression for the PKS system coupled with a fluid flow that obeys Darcy's law for incompressible porous media via buoyancy force \cite{Hu0}; see also \cite{Hu2023}.
	For the 3D coupled PKS-NS system,  Cui-Wang-Wang \cite{cui3} considered the Couette flow  in a finite channel, and obtained the solutions of the 3D linearized PKS-NS system are global as long as the initial cell mass is sufficiently small.
	
	However, for $ \phi(n)=1 $ and $ \psi(n)=n(n+1)^{\alpha} $ with $ \alpha>0 $,  up to now it is still unknown whether the blow-up does not happen provided that the amplitude of some shear flow is sufficiently large. In this paper, we consider $ \phi(n)=1, $ $ \psi(n)=n(n+1) $ and the main goal is to investigate these types of issues.
	
	Before expressing our main theorem, we first introduce a perturbation $v$ around the two-dimensional Couette flow $(Ay,0)$, which $v(t,x,y)=u(t,x,y)-(Ay,0)$ satisfying $ v|_{t=0}=v_{\rm in}=(v_{1,\rm in}, v_{2,\rm in}) $.  Then we rewrite the system (\ref{ini-}) into
	\begin{equation}\label{ini1}
		\left\{
		\begin{array}{lr}
			\partial_tn+Ay\partial_x n+v\cdot\nabla n-\triangle n=-\nabla\cdot(n^2\nabla c)-\nabla(n\nabla c), \\
			\triangle c+n-c=0, \\
			\partial_tv+Ay\partial_x v+\left(
			\begin{array}{c}
				Av_2 \\
				0 \\
			\end{array}
			\right)
			+v\cdot\nabla v-\triangle v+\nabla P=\left(
			\begin{array}{c}
				0 \\
				n \\
			\end{array}
			\right), \\
			\nabla \cdot v=0.
		\end{array}
		\right.
	\end{equation}
	To deal with the pressure $ P, $ we introduce the vorticity $ \omega $ and the stream function $ \Phi $ as follows
	$$  \omega:=\partial_{y}v_{1}-\partial_{x}v_{2},\quad v=\nabla^{\bot}\Phi=(\partial_{y}\Phi, -\partial_{x}\Phi).  $$
	Then $\omega$ satisfies
	\begin{equation*}
		\partial_{t}\omega+Ay\partial_{x}\omega-\triangle\omega+v\cdot\nabla\omega=-\partial_{x} n.
	\end{equation*}
	
	After the time re-scaling $t\mapsto\frac{t}{A}$, we get
	\begin{equation}\label{ini2}
		\left\{
		\begin{array}{lr}
			\partial_tn+y\partial_x n-\frac{1}{A}\triangle n=-\frac{1}{A}\nabla\cdot(v n)-\frac{1}{A}\nabla\cdot(n^2\nabla c)-\frac{1}{A}\nabla\cdot(n\nabla c), \\
			\triangle c+n-c=0, \\
			\partial_t\omega+y\partial_x\omega-\frac{1}{A}\triangle\omega=
			-\frac{1}{A}(\partial_{x}n+v\cdot\nabla\omega), \\
			\nabla \cdot v=0.	\\
		\end{array}
		\right.
	\end{equation}

	Our main result  is stated as follows.
	\begin{theorem}\label{result}
		Assume that the initial data $n_{\rm in}\in L^{1}\cap L^{\infty}(\mathbb{T}\times\mathbb{R})$, $v_{\rm in}\in H^1(\mathbb{T}\times\mathbb{R})$ and the initial cell mass $ M=\|n_{\rm in}\|_{L^{1}} $ satisfies $ C_{*}^{4}M^{2}<3 $ , where $ C_{*} $ is the sharp Sobolev constant of 
		\begin{equation}\label{sharp}
			\|f\|_{L^4(\mathbb{R})}\leq C_{*} \|f\|_{L^1(\mathbb{R})}^{\frac12}\|\partial_{y} f\|_{L^2(\mathbb{R})}^{\frac12}.	
		\end{equation}
		Then there exists a positive constant $A_{1}$ depending on $||n_{\rm in}||_{L^{1}\cap L^{\infty}(\mathbb{T}\times\mathbb{R})}$
		and $||v_{\rm in}||_{H^1(\mathbb{T}\times\mathbb{R})}$, such that if $A\geq A_{1}$,	the solution of (\ref{ini2}) is global in time.
	\end{theorem}
	\begin{remark}\label{interpolation} For  $ \phi(n)=1 $ and $ \psi(n)=n(n+1)^{\alpha} $ with a critical Fujita exponent $ \alpha>0 $ in the system (\ref{ini}),   
as shown by Winkler \cite{winkler0} or \cite{CS2}, for any given mass, there exists a corresponding solution of the same mass
that blows up in either finite or infinite time. Theorem \ref{result} shows that the Couette flow (sufficiently large) can suppress the blow-up of a positive initial mass for $\alpha=1$ and the constant $C_{*}$ comes from the estimate of $ \|n_{0}\|_{L^{\infty}L^{2}} $ (see \eqref{n32}), which seems to be exactly necessary.  Hence, it is still unknown whether the Couette flow  can suppress the blow-up  for the case of $\alpha>1$.

		On the other hand, one can compute the detailed value of the above possible critical initial mass. For the following Gagliardo-Nirenberg interpolation inequality
		\begin{equation*}
			\left\{
			\begin{array}{lr}
				\|f\|_{L^{m+1}(\mathbb{R}^{n})}\leq C_{q,m,p}\|f\|_{L^{q+1}(\mathbb{R}^{n})}^{1-\theta}\|\nabla f\|_{L^{p}(\mathbb{R}^{n})}^{\theta}, \\
				\theta=\frac{pn(m-q)}{(m+1)[n(p-q-1)+p(q+1)]},	\\
			\end{array}
			\right.
		\end{equation*}
		Nagy \cite{nagy 1941} (see also \cite{1993,liu2017}) obtained the best constant $ C_{q,m,p} $ for the one-dimensional case in 1941. In particular, $ C_{0,3,2}=\left(\frac{4\pi^{2}}{9} \right)^{-\frac14} $ for $ p=2, q=0, m=3 $ and $ n=1, $ which corresponds to (\ref{sharp}). Thus, we can give a specific upper bound for the initial cell mass $ M$ as  follows.	
	\end{remark}
	
	\begin{theorem}\label{result 2}
		Assume that the initial data $n_{\rm in}\in L^{1}\cap L^{\infty}(\mathbb{T}\times\mathbb{R})$, $v_{\rm in}\in H^1(\mathbb{T}\times\mathbb{R})$ and the initial cell mass $ M=\|n_{\rm in}\|_{L^{1}}<\frac{2\pi}{\sqrt{3}} $.
		Then there exists a positive constant $A_{1}$ depending on $||n_{\rm in}||_{L^{1}\cap L^{\infty}(\mathbb{T}\times\mathbb{R})}$
		and $||v_{\rm in}||_{H^1(\mathbb{T}\times\mathbb{R})}$, such that if $A\geq A_{1}$,	the solution of (\ref{ini2}) is global in time.
	\end{theorem}
	
\begin{remark}
 Theorem \ref{result 2} shows that the Couette flow (sufficiently large) can suppress the blow-up of a positive initial mass, whose upper bound is greater than $\frac{2\pi}{\sqrt{3}}.$ It's interesting whether the initial mass is sharp. 
\end{remark}
	
	Finally, we will list some notations used in this paper.
	
	\noindent\textbf{Notations}:
	\begin{itemize}
		
		\item Define the Fourier transform by
		\begin{equation}
			f(t,x,y)=\sum_{k\in\mathbb{Z}}\hat{f}(t,y){\rm e}^{ikx}, \nonumber
		\end{equation}
		where $\hat{f}(t,y)=\frac{1}{|\mathbb{T}|}\int_{\mathbb{T}}{f}(t,x,y){\rm e}^{-ikx}dx.$
		
		\item For a given function $f=f(t,x,y)$,   write its $x$-part zero mode and  $x$-part non-zero mode by
		$$P_0f=f_0=\frac{1}{|\mathbb{T}|}\int_{\mathbb{T}}f(t,x,y)dx,\ {\rm and}\ P_{\neq}f=f_{\neq}=f-f_0.$$
		Especially, we use $v_{k,0}$ and $v_{k,\neq}$ to represent the zero mode
		and non-zero mode of the velocity $v_k(k=1,2)$, respectively.
		Similarly,  $\omega_{k,0}$ and $\omega_{k,\neq}$ represent the zero mode
		and non-zero mode of the vorticity $\omega_k(k=1,2)$.
		
		\item The norm of the $L^p$ space and the time-space norm  $\|f\|_{L^qL^p}$ are defined as
		\begin{equation*}
			\|f\|_{L^p(\mathbb{T}\times\mathbb{R})}=\left(\int_{\mathbb{T}\times\mathbb{R}}|f|^p dxdy\right)^{\frac{1}{p}},\quad \|f\|_{L^qL^p}=\big\|  \|f\|_{L^p(\mathbb{T}\times\mathbb{R})}\ \big\|_{L^q(0,t)}.
		\end{equation*}
		
		\item For $ a\geq 0, $ define the norm $\|f\|_{X_{a}}$ by
		\begin{equation*}
			\begin{aligned}
				\|f\|_{X_a}^{2}=&\| e^{aA^{-\frac{1}{3}}t}f\|_{L^\infty L^2}^2+\| e^{aA^{-\frac{1}{3}}t}\nabla^\bot\triangle^{-1}\partial_xf\|_{L^2 L^2}^2 \\
				& +\frac{1}{A^{\frac{1}{3}}}\| e^{aA^{-\frac{1}{3}}t}f\|_{L^2 L^2}^2+\frac{1}{A}\|e^{aA^{-\frac{1}{3}}t}\nabla f\|_{L^2 L^2}^2.
			\end{aligned}
		\end{equation*}	
		
		\item Denote by $ M $ the total mass $ \|n(t)\|_{L^{1}}. $ Clearly, integration by parts and divergence theorem yield that $$  M:=\|n(t)\|_{L^{1}}=\|n_{\rm in}\|_{L^{1}}.$$
		
		\item Throughout this paper,  denote by $ C $ a positive constant independent of $A$, $t$ and the initial data, and it may be different from line to line.
	\end{itemize}
	
	The rest part of this paper is organized as follows. In Section \ref{sec 2}, the key idea and the proof of \textbf{Theorem \ref{result}} are presented. Section \ref{sec 3} is devoted to providing a collection of elementary
	lemmas including space-time estimates, which are essential for the proof of \textbf{Proposition \ref{prop 1}}-\textbf{\ref{prop 2}}.
	In Section \ref{sec 4}, we finish the proof of \textbf{Proposition \ref{prop 1}}, where the smallness of the initial mass is necessary.
	The proof of \textbf{Proposition \ref{prop 2}} is established in Section \ref{sec 6}. In Section \ref{theorem}, we complete the proof of {\textbf{Theorem \ref{result 2}}}.
	
	\section{Key ideas and proof of Theorem \ref{result} }\label{sec 2} 
	For given functions $ f $ and $ g, $ there hold
	\begin{equation*}
		\begin{aligned}
			\left(fg\right)_{0}=f_{0}g_{0}+(f_{\neq}g_{\neq})_{0},\quad
			(fg)_{\neq}=f_0g_{\neq}+f_{\neq}g_0+(f_{\neq}g_{\neq})_{\neq}.
		\end{aligned}
	\end{equation*}
	As the enhanced dissipation of fluid only affects the non-zero mode, and it is essential to separate the zero mode and the non-zero mode of system (\ref{ini2}) as follows:
	\begin{equation}\label{eq:zero mode}
		\left\{\begin{array}{l}
			\partial_t n_0-\frac{1}{A}\triangle n_0=-\frac{1}{A}\nabla\cdot(v_{\neq} n_{\neq})_0-\frac{1}{A}\left(\partial_{y}(n_{0}^{2}\partial_{y}c_{0})+\partial_{y}((n_{\neq}^{2})_{0}\partial_{y}c_{0}) \right)
			\\
			\qquad\qquad\qquad\qquad\qquad-\frac{2}{A}\nabla\cdot(n_{\neq}n_{0}\nabla c_{\neq})_{0}-\frac{1}{A}\nabla\cdot((n_{\neq}^{2})_{\neq}\nabla c_{\neq})_{0} \\\qquad\qquad\qquad\qquad\qquad-\frac{1}{A}\partial_{y}(n_{0}\partial_{y}c_{0})-\frac{1}{A}\nabla\cdot(n_{\neq}\nabla c_{\neq})_{0}, \\
			-\triangle c_0+c_0=n_0 , \\
			\partial_t\omega_0-\frac{1}{A}\triangle\omega_0=-\frac{1}{A}\nabla\cdot(v_{\neq}\omega_{\neq})_0, \\
			v_0=(\partial_y(\partial_{yy})^{-1}\omega_0,0), \\ 
		\end{array}\right.
	\end{equation}
	and
	{\small
		\begin{equation}\label{eq:non-zero mode}
			\left\{\begin{array}{l}
				\partial_t n_{\neq}+y\partial_x n_{\neq}-\frac{1}{A}\triangle n_{\neq}=-\frac{1}{A}\left(\nabla\cdot(v_0n_{\neq})+\nabla\cdot(v_{\neq}n_0)+\nabla\cdot(v_{\neq}n_{\neq})_{\neq}\right)      \\
				\qquad\qquad\qquad\quad-\frac{2}{A}\partial_{y}(n_{\neq}n_{0}\partial_{y}c_{0})-\frac{1}{A}\left(\partial_{y}((n_{\neq}^{2})_{\neq}\partial_{y}c_{0})+\nabla\cdot(n_{0}^{2}\nabla c_{\neq}) \right)
				\\
				\qquad\qquad\qquad\quad
				-\frac{1}{A}\left(\nabla\cdot((n_{\neq}^{2})_{0}\nabla c_{\neq})+\nabla\cdot((n_{\neq}^{2})_{\neq}\nabla c_{\neq})_{\neq}\right)-\frac{2}{A}\nabla\cdot(n_{\neq}n_{0}\nabla c_{\neq})_{\neq}\\
				\qquad\qquad\qquad\quad-\frac{1}{A}\left(\nabla\cdot(n_{0}\nabla c_{\neq})+\partial_{y}(n_{\neq}\partial_{y}c_{0})+\nabla\cdot(n_{\neq}\nabla c_{\neq})_{\neq} \right) ,                                                                \\
				-\triangle c_{\neq}+c_{\neq}=n_{\neq},
				\\
				\partial_t\omega_{\neq}+y\partial_x\omega_{\neq}-\frac{1}{A}\triangle\omega_{\neq}=-\frac{1}{A}\partial_xn_{\neq}-\frac{1}{A}\left(\nabla\cdot(v_0\omega_{\neq})+\nabla\cdot(v_{\neq}\omega_0)+\nabla\cdot(v_{\neq}\omega_{\neq})_{\neq}\right),                               \\
				v_{\neq}=\nabla^\bot\triangle^{-1}\omega_{\neq}.
			\end{array}\right.
		\end{equation}
	}
	\begin{remark}\label{remark 1}
		Note that zero mode and non-zero mode can be controlled by their own functions. That is, for any $ 1\leq p\leq \infty, $ there hold
		\begin{equation*}
			\begin{aligned}
				\|f_{0}\|_{L^{p}(\mathbb{T}\times\mathbb{R})}\leq\|f\|_{L^{p}(\mathbb{T}\times\mathbb{R})} ,
			\end{aligned}
		\end{equation*}
		and
		\begin{equation*}
			\|f_{\neq}\|_{L^{p}(\mathbb{T}\times\mathbb{R})}\leq\|f\|_{L^{p}(\mathbb{T}\times\mathbb{R})}+\|f_{0}\|_{L^{p}(\mathbb{T}\times\mathbb{R})}\leq 2\|f\|_{L^{p}(\mathbb{T}\times\mathbb{R})}.
		\end{equation*}	
	\end{remark}	
	
	\begin{remark}
		Due to $ {\rm div}v_{0}=0, $ we have $ \partial_{y}v_{2,0}=0 $ and $ v_{2,0}=0. $ Then the zero mode of $ v_{1} $ satisfies
		\begin{equation}\label{v0}
			\partial_{t}v_{1,0}-\frac{1}{A}\partial_{yy}v_{1,0}=-\frac{1}{A}\partial_{y}(v_{2,\neq}v_{1,\neq})_{0}.
		\end{equation}
	\end{remark}
	
	We introduce an energy functional:
	\begin{equation*}
		E(t):=\|\omega_{\neq}\|_{X_{a}}+\|n_{\neq}\|_{X_{a}},
	\end{equation*}
	with the initial norm
	\begin{equation*}
		E_{\rm in}:=\|\omega_{\rm in,\neq}\|_{L^{2}}+\|n_{\rm in,\neq}\|_{L^{2}}.
	\end{equation*}	
	Let  $T$ be the terminal point of the largest range $[0, T]$ such that the following hypothesis hold
	\begin{align}
		E(t)\leq 2K_{\neq}, \label{assumption_0}\\
		||n||_{L^{\infty}L^{\infty}}\leq 2K_{\infty},  \label{assumption_1}
	\end{align}
	for any $t\in[0, T]$, where $K_{\neq}$ and $K_{\infty}$ will be calculated during the calculation.

	The following propositions are key to obtaining the main results. Combining them with the local well-posedness of the system (\ref{ini2}), we can deduce the global existence of the solution.
	\begin{proposition}\label{prop 1}
		Assume that the initial date $n_{\rm in}\in L^1\cap L^{\infty}(\mathbb{T}\times\mathbb{R})$ and $v_{\rm in}\in H^1(\mathbb{T}\times\mathbb{R})$, under the conditions of (\ref{assumption_0}) and (\ref{assumption_1}), there exist a positive constant $ K_{\neq} $ depending on $ E_{\rm in}, $ and a positive constant $ A_{2}$ depending on $K_{\neq}, K_{\infty}, M $ and $\|v_{\rm in}\|_{H^{1}} $, such that if $ A\geq A_{2}, $ there holds
		\begin{equation*}
			\begin{aligned}
				E(t)&\leq K_{\neq},
			\end{aligned}
		\end{equation*}
		for all $ t\in[0,T]. $
	\end{proposition}
	
	\begin{proposition}\label{prop 2}
		Assume that the initial date $n_{\rm in}\in L^1\cap L^{\infty}(\mathbb{T}\times\mathbb{R})$, $v_{\rm in}\in H^1(\mathbb{T}\times\mathbb{R})$ and $ C_{*}^{4}M^{2}<3 $, under the conditions of (\ref{assumption_0}) and (\ref{assumption_1}), there exists a positive constant $ K_{\infty} $ depending on $ \|n_{\rm in}\|_{L^{1}\cap L^{\infty}}, $ such that 
		\begin{equation*}
			\begin{aligned}
				\|n\|_{L^{\infty}L^{\infty}}&\leq K_{\infty},
			\end{aligned}
		\end{equation*}
		for all $ t\in[0,T]. $
	\end{proposition}
	\begin{proof}[Proof of Theorem \ref{result}]
		Taking $ A_{1}=\max\{A_{2}, A_{3}\} $ and combining {\textbf{Proposition \ref{prop 1}}} and {\textbf{Proposition \ref{prop 2}}},  the proof is complete.	
	\end{proof}

	\section{A Priori estimates}\label{sec 3}
	
	\subsection{Space-time estimate}
	The following space-time estimate plays an important role to bound the non-zero modes of the solution to the system (\ref{ini2}).
	\begin{lemma}[See Proposition 3.1, \cite{zeng}]\label{lem:space-time}
		Let $ f $ satisfies
		\begin{equation}\label{eq linear}
			\partial_tf+y\partial_xf-\frac{1}{A}\triangle f= \partial_xf_1+f_2+{\rm div} f_3.
		\end{equation}
		If $P_0f=P_0f_1=P_0f_2=P_0f_3=0,$ then for $a\in[0,4],$ 
		we have
		\begin{equation*}
			\begin{aligned}
				\|f\|_{X_a}^2\leq C\left(\|f(0)\|_{L^2}^2+\|e^{aA^{-\frac{1}{3}}t}\nabla f_1\|_{L^2L^2}^2+A^{\frac{1}{3}}\|e^{aA^{-\frac{1}{3}}t} f_2\|_{L^2 L^2}^2+A\|e^{aA^{-\frac{1}{3}}t} f_3\|_{L^2 L^2}^2\right).
			\end{aligned}
		\end{equation*}
	\end{lemma}

	\subsection{Elliptic estimates}
	We estimate $c$ by elliptic energy method, which is similar as in \cite{cui1} or \cite{zeng}.
	\begin{lemma}\label{ellip_0}
		Let $c_0$ and $n_{0}$ be the zero mode of $c$ and $n$, respectively, satisfying
		$$-\partial_{yy} c_0+c_0=n_{0}.$$
		Then there hold
		\begin{align}
			\|\partial_{y}^{2} c_0(t)\|_{L^2}+\|\partial_{y} c_0(t)\|_{L^2}+\|c_0\|_{L^2}
			\leq C\|n_{0}(t)\|_{L^2},\nonumber 
		\end{align}
		$$\|\partial_{y} c_0(t)\|_{L^4}\leq C\|n_{0}(t)\|_{L^2},$$
		and
		$$ \|c_{0}(t)\|_{L^{\infty}}\leq C\|n_{0}(t)\|_{L^{2}},\quad\|\partial_{y}c_{0}(t)\|_{L^{\infty}}\leq C\|n_{0}(t)\|_{L^{2}}, $$
		for any $t\geq0$.
	\end{lemma}
	\begin{proof}
		The basic energy estimates yield
		\begin{equation}
			\begin{aligned}
				\|\partial_{y}^{2} c_0(t)\|^2_{L^2}+\|\partial_{y} c_0(t)\|^2_{L^2}+\|c_0(t)\|^2_{L^2}
				\leq C\|n_{0}(t)\|^2_{L^2}.
				\nonumber
			\end{aligned}
		\end{equation}
		Furthermore, using the Gagliardo-Nirenberg inequality, we have	
		$$\|\partial_{y}c_{0}(t)\|_{L^{4}}\leq C\|\partial_{y}c_{0}(t)\|_{L^{2}}^{\frac34}\|\partial_{y}^{2}c_{0}(t)\|_{L^{2}}^{\frac14}\leq C\|n_{0}(t)\|_{L^{2}}, $$
		$$\|c_{0}(t)\|_{L^{\infty}}\leq C\|c_{0}(t)\|_{L^{2}}^{\frac12}\|\partial_{y}c_{0}(t)\|_{L^{2}}^{\frac12} \leq C\|n_{0}(t)\|_{L^{2}},$$
		and
		$$ \|\partial_{y}c_{0}(t)\|_{L^{\infty}}\leq C\|\partial_{y}c_{0}(t)\|_{L^{2}}^{\frac12}\|\partial_{y}^{2}c_{0}(t)\|_{L^{2}}^{\frac12}\leq C\|n_{0}(t)\|_{L^{2}}. $$
	\end{proof}

	\begin{lemma}\label{ellip_2}
		Let $c_{\neq}$ and $n_{\neq}$ be the non-zero mode of $c$ and $n$,
		respectively, satisfying
		$$-\triangle c_{\neq}+c_{\neq}=n_{\neq}.$$
		Then there hold
		\begin{align}
			\|\triangle c_{\neq}(t&)\|_{L^2}
			+\|\nabla c_{\neq}(t)\|_{L^2}\leq C\|n_{\neq}(t)\|_{L^2},\nonumber
		\end{align}
		and
		\begin{equation}
			\|\nabla c_{\neq}(t)\|_{L^4}\leq C\|n_{\neq}(t)\|_{L^2},\nonumber
		\end{equation}
		for any $t\geq0$.
	\end{lemma}
	\begin{proof}
		By integrating by parts, we have
		\begin{equation}
			\begin{aligned}
				\|\triangle c_{\neq}(t)\|^2_{L^2}+\|\nabla c_{\neq}(t)\|^2_{L^2}
				+\|c_{\neq}(t)\|^2_{L^2}
				&\leq C\|n_{\neq}(t)\|^2_{L^2}.
				\nonumber
			\end{aligned}
		\end{equation}
		Using the Gagliardo-Nirenberg inequality, we obtain
		$$\|\nabla c_{\neq}(t)\|_{L^4}\leq
		C\| c_{\neq}(t)\|^{\frac{1}{4}}_{L^2}
		\|\triangle c_{\neq}(t)\|^{\frac{3}{4}}_{L^2}
		\leq C\|n_{\neq}(t)\|_{L^2}.$$
	\end{proof}
	
	\subsection{A Priori estimates for non-zero mode}
	\begin{lemma}\label{lemma_0}
		Let $f$ be a function such that $f_{\neq}\in H^1(\mathbb{T}\times\mathbb{R})$, there holds
		$$||f_{\neq}||_{L^2(\mathbb{T}\times\mathbb{R})}
		\leq C\|\partial_{x}f_{\neq}\|_{L^{2}(\mathbb{T}\times\mathbb{R})}\leq  C||\nabla f_{\neq}||_{L^2(\mathbb{T}\times\mathbb{R})}.$$
	\end{lemma}
	\begin{proof} It follows from Poincar\'{e}'s inequality immediately and we omit it.
	\end{proof}
	\begin{lemma}\label{lem: non-zero}
		Let $ v_{\neq} $ is determined by $ (\ref{eq:non-zero mode})_{4}. $ Then there holds
		$$\| e^{aA^{-\frac{1}{3}}t}v_{\neq}\|_{L^2L^\infty}\leq CA^{\frac{1}{4}}\|\omega_{\neq}\|_{X_a}.$$
	\end{lemma}
	\begin{proof}
		Recall that $ v=\nabla^{\bot}\Phi=(\partial_{y}\Phi, -\partial_{x}\Phi) $ and $ \triangle\Phi=\omega, $ and using the Fourier series, we get
		\begin{equation*}
			v_{1,\neq}=\sum_{k\neq 0}\widehat{\partial_{y}\Phi}_{\neq}(k,y)e^{ikx},
		\end{equation*}
		\begin{equation*}
			v_{2,\neq}=\sum_{k\neq 0}\widehat{\partial_{x}\Phi}_{\neq}(k,y)e^{ikx}=-\sum_{k\neq 0}ik\widehat{\Phi}_{\neq}(k,y)e^{ikx},
		\end{equation*}
		and
		\begin{equation*}
			\omega_{\neq}=\sum_{k\neq 0}(\partial_{y}^{2}-k^{2})\widehat{\Phi}_{\neq}(k,y)e^{ikx},
		\end{equation*}
		where $ \hat{\Phi}_{\neq}=\frac{1}{2\pi}\int_{\mathbb{T}}\Phi_{\neq}(x,y)e^{-ikx}dx. $
		Then direct calculations indicate that
		\begin{equation}\label{1}
			\begin{aligned}
				\|v_{1,\neq}\|_{L^{2}}^{2}=2\pi\int_{\mathbb{R}}\sum_{k\neq 0}\left(\widehat{\partial_{y}\Phi}_{\neq}(k,y)\overline{\widehat{\partial_{y}\Phi}}_{\neq}(k,y) \right)dy=2\pi\sum_{k\neq 0}\|\widehat{\Phi}_{\neq}'(k,\cdot)\|_{L^{2}}^{2},
			\end{aligned}
		\end{equation}
		\begin{equation}\label{2}
			\begin{aligned}
				\|v_{2,\neq}\|_{L^{2}}^{2}=2\pi\int_{\mathbb{R}}\sum_{k\neq 0}\left(|k|^{2}\widehat{\Phi}_{\neq}(k,y)\overline{\widehat{\Phi}}_{\neq}(k,y) \right)dy=2\pi\sum_{k\neq 0}|k|^{2}\|\widehat{\Phi}_{\neq}(k,\cdot)\|_{L^{2}}^{2},
			\end{aligned}
		\end{equation}
		and
		\begin{equation}\label{3}
			\begin{aligned}
				\|\omega_{\neq}\|_{L^{2}}^{2}=&2\pi\int_{\mathbb{R}}\sum_{k\neq 0}\left((\partial_{y}^{2}-k^{2})\widehat{\Phi}_{\neq}(k,y)(\partial_{y}^{2}-k^{2})\overline{\widehat{\Phi}}_{\neq}(k,y) \right)dy\\=&\sum_{k\neq 0}2\pi\left(|k|^{4}\|\widehat{\Phi}_{\neq}(k,\cdot)\|_{L^{2}}^{2}+2|k|^{2}\|\widehat{\Phi}_{\neq}'(k,\cdot)\|_{L^{2}}^{2}+\|\widehat{\Phi}_{\neq}''(k,\cdot)\|_{L^{2}}^{2} \right).
			\end{aligned}
		\end{equation}
		It follows from (\ref{1})-(\ref{3}) that
		\begin{equation}\label{4}
			\|\nabla v_{\neq}\|_{L^{2}}\leq C\|\omega_{\neq}\|_{L^{2}}.
		\end{equation}
		
		Consider that $ v_{\neq}=\sum_{k\neq 0}\widehat{v}_{\neq}(k,y)e^{ikx}, $ using the Gagliardo-Nirenberg inequality, we obtain
		\begin{equation*}
			\|v_{\neq}\|_{L^{\infty}}\leq \sum_{k\neq 0}\|\widehat{v}_{\neq}(k,y)\|_{L_{y}^{\infty}}\leq C\sum_{k\neq 0}\|\widehat{v}_{\neq}(k,\cdot)\|_{L^{2}}^{\frac12}\|\partial_{y}\widehat{v}_{\neq}(k,\cdot)\|_{L^{2}}^{\frac12}.
		\end{equation*}
		Due to H$\ddot{\mathrm{o}}$lder's inequality, there holds
		\begin{equation*}
			\begin{aligned}
				\|v_{\neq}\|_{L^{\infty}}\leq&C\left(\sum_{k\neq 0}|k|^{1+2\varepsilon}||\widehat{v}_{\neq}(k,\cdot)\|_{L^{2}}\|\partial_{y}\widehat{v}_{\neq}(k,\cdot)\|_{L^{2}} \right)^{\frac12}\left(\sum_{k\neq 0}\frac{1}{|k|^{1+2\varepsilon}} \right)^{\frac12}\\
\leq&C\left(\sum_{k\neq 0}|k|^{2\varepsilon}\|\widehat{\nabla v}_{\neq}(k,\cdot)\|_{L^{2}}^{2} \right)^{\frac12},
			\end{aligned}
		\end{equation*}
		where $ \varepsilon\in (0,\frac12]. $ Using  H$\ddot{\mathrm{o}}$lder's inequality again and (\ref{4}), we have
		\begin{equation*}
			\begin{aligned}
				\|v_{\neq}\|_{L^{\infty}}^{2}\leq&C\sum_{k\neq 0}\|k\widehat{\nabla v}_{\neq}(k,\cdot)\|_{L^{2}}^{2\varepsilon}\|\widehat{\nabla v}_{\neq}(k,\cdot)\|_{L^{2}}^{2(1-\varepsilon)}\\
\leq&C\|\partial_{x}\nabla v_{\neq}\|_{L^{2}}^{2\varepsilon}\|\nabla v_{\neq}\|_{L^{2}}^{2(1-\varepsilon)}
\leq C\|\partial_{x}\omega_{\neq}\|_{L^{2}}^{2\varepsilon}\|\omega_{\neq}\|_{L^{2}}^{2(1-\varepsilon)},
			\end{aligned}
		\end{equation*}
		which follows that
		\begin{equation*}
			\|e^{aA^{-\frac13}t}v_{\neq}\|_{L^{2}L^{\infty}}\leq C\|e^{aA^{-\frac13}t}\partial_{x}\omega_{\neq}\|_{L^{2}L^{2}}^{\varepsilon}\|e^{aA^{-\frac13}t}\omega_{\neq}\|_{L^{2}L^{2}}^{1-\varepsilon}\leq CA^{\frac{1+2\varepsilon}{6}}\|\omega_{\neq}\|_{X_{a}},
		\end{equation*}
		for any $ \varepsilon\in (0,\frac12]. $
		
		We complete the proof by selecting $ \varepsilon=\frac14. $
	\end{proof}
	
	\subsection{A Priori estimates for zero mode}
	\begin{lemma}\label{lem: zero mode}
		Under the assumptions (\ref{assumption_0})-(\ref{assumption_1}) and $ C_{*}^{4}M^{2}<3 $, there exists a positive constant $ A_{3} $ depending on $ K_{\neq}, K_{\infty}  $ and $ M $, such that if $ A\geq A_{3}, $ there hold
		\begin{equation}\label{n0}
			\|n_{0}\|_{L^{\infty}L^{2}}\leq C\left(\|n_{\rm in,0}\|_{L^{2}}+\frac{M^{3}}{(3-C_{*}^{4}M^{2})^{\frac34}}+1\right):=H_{1},
		\end{equation}
		\begin{equation}\label{n0 4}
			\|n_{0}\|_{L^{\infty}L^{4}}\leq C\left(\|n_{\rm in,0}\|_{L^{4}}+\|n_{\rm in,0}\|_{L^{2}}^{\frac32}+\frac{M^{\frac92}}{(3-C_{*}^{4}M^{2})^{\frac98}}+1 \right):=H_{2},
		\end{equation}
		\begin{equation}\label{omega 0}
			\|\omega_0\|_{L^\infty L^2}+A^{-\frac12}\|\partial_y\omega_0\|_{L^2L^2}\leq C\left(\|\omega_{\rm in,0}\|_{L^{2}}+1\right):=H_{3},
		\end{equation}
		\begin{equation}\label{v01}
			\|v_{1,0}\|_{L^\infty L^\infty}\leq C\left(\|v_{\rm in,0}\|_{L^{2}}+\|\omega_{\rm in,0}\|_{L^{2}}+1 \right):=H_{4}.
		\end{equation}		
	\end{lemma}
	\begin{proof}
		\underline{\textbf{I. Estimate of $ \|n_{0}\|_{L^{\infty}L^{2}}. $}}
		Multiplying both sides of $(\ref{eq:zero mode})_{1}$ by $2n_0$ and integrating it with $ y $ over $\mathbb{R}$, we obtain
		\begin{equation}\label{energy n0}
			\begin{aligned}
				&\dfrac{d}{dt}\| n_0\|_{L^2}^2+\dfrac{2}{A}\|\partial_yn_0\|_{L^2}^2\\
				=&\dfrac{2}{A}\int_{\mathbb{R}}(v_{\neq}n_{\neq})_0\nabla n_0dy+\dfrac{2}{A}\int_{\mathbb{R}}n^2_0\partial_yc_0\partial_yn_0dy+\dfrac{2}{A}\int_{\mathbb{R}}(n_{\neq}^2)_0\partial_{y}c_0\partial_{y}n_0dy\\
				&+\dfrac{4}{A}\int_{\mathbb{R}}(n_{\neq}n_0\partial_{y}c_{\neq})_0\partial_{y}n_0dy+\dfrac{2}{A}\int_{\mathbb{R}}((n^2_{\neq})_{\neq}\partial_yc_{\neq})_0\partial_yn_0dy\\&+\frac{2}{A}\int_{\mathbb{R}}n_{0}\partial_{y}c_{0}\partial_{y}n_{0}dy+\frac{2}{A}\int_{\mathbb{R}}(n_{\neq}\nabla c_{\neq})_{0}\partial_{y}n_{0}dy\\
				\leq&\dfrac{5\delta}{A}\|\partial_yn_0\|_{L^2}^2+\dfrac{C(\delta)}{A}\|(v_{\neq}n_{\neq})_0\|_{L^2}^2+\dfrac{C(\delta)}{A}\|(n_{\neq}^2)_0\partial_{y}c_0\|_{L^2}^2\\
				&+\dfrac{C(\delta)}{A}\|(n_{\neq}n_0\partial_{y}c_{\neq})_0\|_{L^2}^2+\frac{C(\delta)}{A}\|((n^2_{\neq})_{\neq}\partial_{y}c_{\neq})_{0}\|_{L^{2}}^{2}+\frac{C(\delta)}{A}\|(n_{\neq}\nabla c_{\neq})_{0}\|_{L^{2}}^{2}\\&+\dfrac{2}{A}\int_{\mathbb{R}}n^2_0\partial_yc_0\partial_yn_0dy+\frac{2}{A}\int_{\mathbb{R}}n_{0}\partial_{y}c_{0}\partial_{y}n_{0}dy,
			\end{aligned}
		\end{equation}
		where $ \delta $ is a positive constant.	
		Due to integration by parts and $(\ref{eq:zero mode})_2$, there holds
		\begin{equation}\label{n32}
			\begin{aligned}
				&\dfrac{2}{A}\int_{\mathbb{R}}n^2_0\partial_yc_0\partial_yn_0dy+\frac{2}{A}\int_{\mathbb{R}}n_{0}\partial_{y}c_{0}\partial_{y}n_{0}dy\\=&-\dfrac{2}{3A}\int_{\mathbb{R}}n_0^3\partial_{yy}c_0dy-\frac{1}{A}\int_{\mathbb{R}}n_{0}^{2}\partial_{yy}c_{0}dy\\
				=&\dfrac{2}{3A}\int_{\mathbb{R}}n_0^4dy-\dfrac{2}{3A}\int_{\mathbb{R}}c_0n_0^3dy+\frac{1}{A}\int_{\mathbb{R}}n_{0}^{3}dy-\frac{1}{A}\int_{\mathbb{R}}c_{0}n_{0}^{2}dy\\:=&I_{1}+I_{2}+I_{3}+I_{4}.
			\end{aligned}
		\end{equation}
		Using the Gagliardo-Nirenberg inequality
		\begin{equation}\label{c*}
			\|n_{0}\|_{L^{4}}\leq C_{*}\|n_{0}\|_{L^{1}}^{\frac12}\|\partial_{y}n_{0}\|_{L^{2}}^{\frac12},\quad\|c_{0}\|_{L^{\infty}}\leq C \|c_{0}\|_{L^{2}}^{\frac12}\|\partial_{y}c_{0}\|_{L^{2}}^{\frac12},
		\end{equation}
		$$ \|n_{0}\|_{L^{3}}\leq C\|n_{0}\|_{L^{2}}^{\frac56}\|\partial_{y}n_{0}\|_{L^{2}}^{\frac16}, $$
		and Young's inequality, {\textbf{Lemma \ref{1}}}, we get
		\begin{equation*}
			I_{1}\leq \dfrac{2}{3A}C_{*}^{4}\|n_{0}\|_{L^1}^2\|\partial_{y}n_{0}\|_{L^{2}}^2\leq \dfrac{2}{3A}C_{*}^{4}M^2\|\partial_{y}n_{0}\|_{L^{2}}^2,
		\end{equation*}
		\begin{equation*}
			\begin{aligned}
				I_{2}\leq&\frac{2}{3A}\|c_{0}\|_{L^{\infty}}\|n_{0}\|_{L^{3}}^{3}\leq\frac{C}{A}\|c_{0}\|_{L^{2}}^{\frac12}\|\partial_{y}c_{0}\|_{L^{2}}^{\frac12}\|n_{0}\|_{L^{2}}^{\frac52}\|\partial_{y}n_{0}\|_{L^{2}}^{\frac12}\\\leq&\frac{C}{A}\|n_{0}\|_{L^{2}}^{\frac72}\|\partial_{y}n_{0}\|_{L^{2}}^{\frac12}\leq\frac{\delta}{A}\|\partial_{y}n_{0}\|_{L^{2}}^{2}+\frac{C(\delta)}{A}\|n_{0}\|_{L^{2}}^{\frac{14}{3}},
			\end{aligned}
		\end{equation*}	
		\begin{equation*}
			I_{3}\leq \frac{C}{A}\|n_{0}\|_{L^{2}}^{\frac52}\|\partial_{y}n_{0}\|_{L^{2}}^{\frac12}\leq\frac{\delta}{A}\|\partial_{y}n_{0}\|_{L^{2}}^{2}+\frac{C(\delta)}{A}\|n_{0}\|_{L^{2}}^{\frac{10}{3}},
		\end{equation*}
		and
		\begin{equation*}
			I_{4}\leq \frac{1}{A}\|c_{0}\|_{L^{\infty}}\|n_{0}\|_{L^{2}}^{2}\leq\frac{C}{A}\|n_{0}\|_{L^{2}}^{3}.
		\end{equation*}
		Combining with $I_{1}-I_{4}$, (\ref{n32}) yields that
		\begin{equation*}
			\begin{aligned}
				&\frac{2}{A}\int_{\mathbb{R}}n^2_{0}\partial_{y}c_{0}\partial_{y}n_{0}dy+\frac{2}{A}\int_{\mathbb{R}}n_{0}\partial_{y}c_{0}\partial_{y}n_{0}dy\\\leq&\frac{C(\delta)}{A}\left(\|n_{0}\|_{L^{2}}^{\frac{14}{3}}+\|n_{0}\|_{L^{2}}^{\frac{10}{3}}+\|n_{0}\|_{L^{2}}^{3} \right)+\left(2\delta+\frac23C_{*}^{4}M^{2} \right)\frac{\|\partial_{y}n_{0}\|_{L^{2}}^{2}}{A},
			\end{aligned}
		\end{equation*}
		substituting it into (\ref{energy n0}), we get
		\begin{equation}\label{energy n0 1}
			\begin{aligned}
				\frac{d}{dt}\|n_{0}\|_{L^{2}}^{2}\leq&-\left(2-7\delta-\frac23C_{*}^{4}M^{2}\right)\frac{\|\partial_{y}n_{0}\|_{L^{2}}^{2}}{A}+\frac{C(\delta)}{A}\|n_{0}\|_{L^{2}}^{\frac{14}{3}}+\frac{C(\delta)}{A}\|n_{0}\|_{L^{2}}^{\frac{10}{3}}\\&+\frac{C}{A}\|n_{0}\|_{L^{2}}^{3}+\frac{C(\delta)}{A}\left(\|v_{\neq}n_{\neq}\|_{L^{2}}^{2}+\|(n_{\neq}^{2})_{\neq}\partial_{y}c_{\neq}\|_{L^{2}}^{2}\right)\\&+\frac{C(\delta)}{A}\left(\|n_{\neq}n_{0}\partial_{y}c_{\neq}\|_{L^{2}}^{2}+\|(n_{\neq}^{2})_{0}\partial_{y}c_{0}\|_{L^{2}}^{2}+\|n_{\neq}\nabla c_{\neq}\|_{L^{2}}^{2} \right).
			\end{aligned}
		\end{equation}
		Notice that $ C_{*}^{4}M^{2}<3 $, then letting 
		$$ 7\delta=1-\frac13 C_{*}^{4}M^{2},\quad \tau=2-7\delta-\frac23C_{*}^{4}M^{2}=1-\frac13 C_{*}^{4}M^{2}>0, $$
		and
		using the Nash inequality
		$$
		-\|\partial_yn_0\|_{L^2}^2\leq-C\dfrac{\| n_0\|_{L^2}^6}{\| n_0\|_{L^1}^4}\leq-C\dfrac{\| n_0\|_{L^2}^6}{M^4},
		$$ 
		it follows from (\ref{energy n0 1}) that
		\begin{equation}\label{energy n0 2}
			\begin{aligned}
				\frac{d}{dt}\|n_{0}\|_{L^{2}}^{2}\leq&-\frac{C\tau}{AM^{4}}\|n_{0}\|_{L^{2}}^{3}\left(\|n_{0}\|_{L^{2}}^{3}-\frac{M^{4}}{\tau}\|n_{0}\|_{L^{2}}^{\frac53}-\frac{M^{4}}{\tau}\|n_{0}\|_{L^{2}}^{\frac13}-\frac{M^{4}}{\tau} \right)\\&+\frac{C}{A}\left(\|v_{\neq}n_{\neq}\|_{L^{2}}^{2}+\|(n_{\neq}^{2})_{\neq}\partial_{y}c_{\neq}\|_{L^{2}}^{2}+\|n_{\neq}n_{0}\partial_{y}c_{\neq}\|_{L^{2}}^{2}\right)\\&+\frac{C}{A}\left(\|(n_{\neq}^{2})_{0}\partial_{y}c_{0}\|_{L^{2}}^{2}+\|n_{\neq}\nabla c_{\neq}\|_{L^{2}}^{2} \right)\\\leq&-\frac{C\tau}{AM^{4}}\|n_{0}\|_{L^{2}}^{3}\left(\|n_{0}\|_{L^{2}}^{3}-\left(\frac{M^{4}}{\tau} \right)^{\frac94}-1 \right)\\&+\frac{C}{A}\left(\|v_{\neq}n_{\neq}\|_{L^{2}}^{2}+\|(n_{\neq}^{2})_{\neq}\partial_{y}c_{\neq}\|_{L^{2}}^{2}+\|n_{\neq}n_{0}\partial_{y}c_{\neq}\|_{L^{2}}^{2}\right)\\&+\frac{C}{A}\left(\|(n_{\neq}^{2})_{0}\partial_{y}c_{0}\|_{L^{2}}^{2}+\|n_{\neq}\nabla c_{\neq}\|_{L^{2}}^{2} \right).
			\end{aligned}
		\end{equation}
		For all $ t\geq 0, $ we denote
		\begin{equation*}
			\begin{aligned}
				G(t):=&\frac{C}{A}\int_{0}^{t}\bigg(\|v_{\neq}n_{\neq}\|_{L^{2}}^{2}+\|(n_{\neq}^{2})_{\neq}\partial_{y}c_{\neq}\|_{L^{2}}^{2}+\|n_{\neq}n_{0}\partial_{y}c_{\neq}\|_{L^{2}}^{2}\\&+\|(n_{\neq}^{2})_{0}\partial_{y}c_{0}\|_{L^{2}}^{2}+\|n_{\neq}\nabla c_{\neq}\|_{L^{2}}^{2} \bigg)ds.
			\end{aligned}
		\end{equation*}
		Due to 	{\textbf{Lemma \ref{lemma_0}}} and (\ref{4}), we find
		\begin{equation}\label{v0 l2}
			\|v_{\neq}\|_{L^2}\leq C\|\nabla v_{\neq}\|_{L^2}\leq C\|\omega_{\neq}\|_{L^2},
		\end{equation}
		and notice that
		\begin{equation}\label{n0 infty 2}
			\|n_{0}\|_{L^{\infty}L^{2}}\leq \|n_{0}\|_{L^{\infty}L^{1}}^{\frac12}\|n_{0}\|_{L^{\infty}L^{\infty}}^{\frac12}\leq CM^{\frac12}K_{\infty}^{\frac12},
		\end{equation}
		then using assumptions (\ref{assumption_0})-(\ref{assumption_1}), {\textbf{Lemma \ref{ellip_0}}} and {\textbf{Lemma \ref{ellip_2}}}, direct calculations indicate that
		\begin{equation*}
			\begin{aligned}
				G(t)\leq&\dfrac{C}{A}\left(\|v_{\neq}\|_{L^2L^2}^2\| n_{\neq}\|_{L^\infty L^\infty}^2+\|n^2_{\neq}\|_{L^\infty L^\infty}^2\|\partial_yc_{\neq}\|_{L^2L^2}^2\right)\\&+\frac{C}{A}\left(\|n_{\neq}n_0\|_{L^{\infty}L^{\infty}}^{2}\|\partial_{y}c_{\neq}\|_{L^{2}L^{2}}^{2}+\|(n_{\neq}^2)_0\|_{L^{2}L^{2}}^2\|\partial_{y}c_0\|_{L^{\infty}L^{\infty}}^2\right)\\&+\frac{C}{A}\|n_{\neq}\|_{L^{\infty}L^{\infty}}^{2}\|\nabla c_{\neq}\|_{L^{2}L^{2}}^{2}\\
				\leq&\dfrac{C}{A}\left(\|\omega_{\neq}\|_{L^2L^2}^2\| n_{\neq}\|_{L^\infty L^\infty}^2+\|n\|_{L^{\infty}L^{\infty}}^{4}\| n_{\neq}\|_{L^{2}L^{2}}^{2}\right)\\
				&+\frac{C}{A}\|n_{\neq}\|_{L^{\infty}L^{\infty}}^2\|n_{\neq}\|_{L^2L^2}^2\|n_0\|_{L^{\infty}L^2}^2+\frac{C}{A}\|n_{\neq}\|_{L^{\infty}L^{\infty}}^{2}\|n_{\neq}\|_{L^{2}L^{2}}^{2}\\
				\leq&\frac{C}{A^{\frac23}}\left(K_{\infty}^{2}\|\omega_{\neq}\|_{X_{a}}^{2}+K_{\infty}^{4}\|n_{\neq}\|_{X_{a}}^{2}+MK_{\infty}^{3}\|n_{\neq}\|_{X_{a}}^{2}+K_{\infty}^{2}\|n_{\neq}\|_{X_{a}}^{2} \right)\\
				\leq&\frac{C}{A^{\frac23}}\left(K_{\neq}^{2}K_{\infty}^{2}+K_{\neq}^{2}K_{\infty}^{4}+MK_{\neq}^{2}K_{\infty}^{3} \right)\leq C,
			\end{aligned}
		\end{equation*}
		provided that 
		$$
		A\geq\left(K_{\neq}^{2}K_{\infty}^{2}+K_{\neq}^{2}K_{\infty}^{4}+MK_{\neq}^{2}K_{\infty}^{3} \right)^\frac32.
		$$
		Then (\ref{energy n0 2}) can be written as follows
		\begin{equation*}
			\begin{aligned}
				\frac{d}{dt}\left(\|n_{0}\|_{L^{2}}^{2}-G(t) \right)\leq-\frac{C\tau}{AM^{4}}\|n_{0}\|_{L^{2}}^{3}\left(\|n_{0}\|_{L^{2}}^{3}-\frac{M^{9}}{\tau^{\frac94}}-1 \right),
			\end{aligned}
		\end{equation*}
		which implies that
		\begin{equation*}
			\begin{aligned}
				\|n_{0}\|_{L^{\infty}L^{2}}\leq& C\left(\|n_{\rm in,0}\|_{L^{2}}+\frac{M^{3}}{\tau^{\frac34}}+1 \right)= C\left(\|n_{\rm in,0}\|_{L^{2}}+\frac{M^{3}}{(3-C_{*}^{4}M^{2})^{\frac34}}+1\right).
			\end{aligned}
		\end{equation*}
		Hence (\ref{n0}) holds.
		
		\underline{\textbf{II. Estimate of $ \|n_{0}\|_{L^{\infty}L^{4}}. $}}
		Multiplying both sides of $(\ref{eq:zero mode})_{1}$ by $4n_0^3$ and integrating it with $ y $ over $\mathbb{R}$, we obtain
		\begin{equation}\label{energy n0l4}
			\begin{aligned}
				&\dfrac{d}{dt}\| n_0\|_{L^4}^4+\dfrac{3}{A}\|\partial_yn_0^{2}\|_{L^2}^2\\
				=&\dfrac{6}{A}\int_{\mathbb{R}}(v_{\neq}n_{\neq})_0n_0\nabla n_0^2dy+\dfrac{4}{A}\int_{\mathbb{R}}n^2_0\partial_yc_0\partial_yn_0^3dy+\dfrac{6}{A}\int_{\mathbb{R}}(n_{\neq}^2)_0\partial_{y}c_0n_0\partial_{y}n_0^2dy\\
				&+\dfrac{12}{A}\int_{\mathbb{R}}(n_{\neq}n_0\partial_{y}c_{\neq})_0n_0\partial_{y}n_0^2dy+\dfrac{6}{A}\int_{\mathbb{R}}((n^2_{\neq})_{\neq}\partial_yc_{\neq})_0n_0\partial_yn_0^2dy\\&+\frac{6}{A}\int_{\mathbb{R}}(n_{\neq}\nabla c_{\neq})_{0}n_{0}\nabla n_{0}^{2}dy+\frac{4}{A}\int_{\mathbb{R}}n_{0}\partial_{y}c_{0}\partial_{y}n_{0}^{3}dy\\
				\leq&\dfrac{1}{A}\|\partial_yn_0^2\|_{L^2}^2+\dfrac{45}{A}\|(v_{\neq}n_{\neq})_0n_0\|_{L^2}^2+\frac{45}{A}\|(n^2_{\neq})_0\partial_{y}c_0n_0\|_{L^{2}}^{2}\\
				&+\dfrac{180}{A}\|(n_{\neq}n_0\partial_{y}c_{\neq})_0n_0\|_{L^2}^2+\dfrac{45}{A}\|((n_{\neq}^2)_{\neq}\partial_{y}c_{\neq})_0n_0\|_{L^2}^2+\frac{45}{A}\|(n_{\neq}\nabla c_{\neq})_{0}n_{0}\|_{L^{2}}^{2}\\&+\dfrac{4}{A}\int_{\mathbb{R}}n^2_0\partial_yc_0\partial_yn_0^3dy+\frac{4}{A}\int_{\mathbb{R}}n_{0}\partial_{y}c_{0}\partial_{y}n_{0}^{3}dy.
			\end{aligned}
		\end{equation}
		Similarly as (\ref{n32}), we find that
		\begin{equation}\label{energy n065}
			\begin{aligned}
				&	\dfrac{4}{A}\int_{\mathbb{R}}n^2_0\partial_yc_0\partial_yn_0^3dy+\frac{4}{A}\int_{\mathbb{R}}n_{0}\partial_{y}c_{0}\partial_{y}n_{0}^{3}dy\\=&-\dfrac{12}{5A}\int_{\mathbb{R}}n_0^5\partial_{yy}c_0dy-\frac{3}{A}\int_{\mathbb{R}}n_{0}^{4}\partial_{yy}c_{0}dy\\
				=&\dfrac{12}{5A}\int_{\mathbb{R}}n_0^6dy-\dfrac{12}{5A}\int_{\mathbb{R}}c_0n_0^5dy+\frac{3}{A}\int_{\mathbb{R}}n_{0}^{5}dy-\frac{3}{A}\int_{\mathbb{R}}c_{0}n_{0}^{4}dy\\:=&J_{1}+J_{2}+J_{3}+J_{4}.
			\end{aligned}
		\end{equation}
		Applying the Gagliardo-Nirenberg inequality, Young's inequality and {\textbf{Lemma \ref{1}}}, there hold
		\begin{equation*}
			\begin{aligned}
				J_{1}=\dfrac{12}{5A}\|n_0^2\|_{L^3}^3\leq\dfrac{C}{A}\|n_0^2\|_{L^2}^{\frac52}\|\partial_{y}n_0^2\|_{L^2}^{\frac12}\leq\dfrac{1}{4A}\|\partial_{y}n_0^2\|_{L^2}^2+\dfrac{C}{A}\|n_0^2\|_{L^2}^{\frac{10}{3}},
			\end{aligned}
		\end{equation*}	
		\begin{equation*}
			\begin{aligned}
				J_{2}&\leq\dfrac{12}{5A}\|c_0\|_{L^\infty}\|n_0^2\|_{L^{\frac52}}^{\frac52}\leq\dfrac{C}{A}\|n_0\|_{L^2}\|n_0^2\|_{L^2}^{\frac94}\|\partial_{y}n_0^2\|_{L^2}^{\frac14}\\
				&\leq\dfrac{1}{4A}\|\partial_{y}n_0^2\|_{L^2}^2+\dfrac{C}{A}\|n_0\|_{L^\infty L^2}^{\frac87}\|n_0^2\|_{L^2}^{\frac{18}{7}},
			\end{aligned}
		\end{equation*}
		\begin{equation*}
			J_{3}\leq\frac{C}{A}\|n_{0}^{2}\|_{L^{2}}^{\frac94}\|\partial_{y}n_{0}^{2}\|_{L^{2}}^{\frac14}\leq\frac{1}{4A}\|\partial_{y}n_{0}^{2}\|_{L^{2}}^{2}+\frac{C}{A}\|n_{0}^{2}\|_{L^{2}}^{\frac{18}{7}},
		\end{equation*}
		and
		\begin{equation*}
			J_{4}\leq \frac{3}{A}\|c_{0}\|_{L^{\infty}}\|n_{0}^{2}\|_{L^{2}}^{2}\leq\frac{C}{A}\|n_{0}\|_{L^{\infty}L^{2}}\|n_{0}^{2}\|_{L^{2}}^{2}.
		\end{equation*}
		Combing with $J_{1}-J_4$, (\ref{energy n065}) follows that
		\begin{equation}\label{energy n0xin}
			\begin{aligned}
				&\dfrac{4}{A}\int_{\mathbb{R}}n^2_0\partial_yc_0\partial_yn_0^3dy+\frac{4}{A}\int_{\mathbb{R}}n_{0}\partial_{y}c_{0}\partial_{y}n_{0}^{3}dy\\\leq&\dfrac{1}{A}\|\partial_{y}n_0^2\|_{L^2}^2+\dfrac{C}{A}\|n_0^2\|_{L^2}^{\frac{10}{3}}+\dfrac{C}{A}(H_1^{\frac87}+1)\|n_0^2\|_{L^2}^{\frac{18}{7}}+\frac{C}{A}H_{1}\|n_{0}^{2}\|_{L^{2}}^{2}.
			\end{aligned}
		\end{equation}  
		Due to the Nash inequality
		\begin{equation*}
			-\|\partial_{y}n_0^2\|_{L^2}^2\leq-C\dfrac{\|n_0^{2}\|_{L^2}^6}{\|n_0^2\|_{L^1}^4}\leq-C\dfrac{\|n_0^2\|_{L^2}^6}{H_1^8},
		\end{equation*} 
		and (\ref{energy n0xin}), we get from  (\ref{energy n0l4}) that
		\begin{equation}\label{energy n04}
			\begin{aligned}   
				&\dfrac{d}{dt}\| n_0\|_{L^4}^4\\
				\leq&-\dfrac{C}{AH_1^8}\|n_0^2\|_{L^2}^6+\dfrac{C}{A}\|n_0^2\|_{L^2}^{\frac{10}{3}}+\dfrac{C}{A}(H_1^{\frac87}+1)\|n_0^2\|_{L^2}^{\frac{18}{7}}+\frac{C}{A}H_{1}\|n_{0}^{2}\|_{L^{2}}^{2}\\&+\dfrac{C}{A}\|(v_{\neq}n_{\neq})_0n_0\|_{L^2}^2+\frac{C}{A}\|(n^2_{\neq})_0\partial_{y}c_0n_0\|_{L^{2}}^{2}+\dfrac{C}{A}\|(n_{\neq}n_0\partial_{y}c_{\neq})_0n_0\|_{L^2}^2\\
				&+\dfrac{C}{A}\|((n_{\neq}^2)_{\neq}\partial_{y}c_{\neq})_0n_0\|_{L^2}^2+\frac{C}{A}\|(n_{\neq}\nabla c_{\neq})_{0}n_{0}\|_{L^{2}}^{2}\\
				\leq&-\dfrac{C\|n_0^2\|_{L^2}^{2}}{AH_1^8}\left(\|n_0^2\|_{L^2}^{4}-H_1^8\|n_0^2\|_{L^2}^{\frac{4}{3}}-(H_1^{\frac{64}{7}}+H_{1}^{8})\|n_{0}^{2}\|_{L^{2}}^{\frac47}-H_{1}^{9}\right)\\&+\dfrac{C}{A}\|(v_{\neq}n_{\neq})_0n_0\|_{L^2}^2+\frac{C}{A}\|(n^2_{\neq})_0\partial_{y}c_0n_0\|_{L^{2}}^{2}+\dfrac{C}{A}\|(n_{\neq}n_0\partial_{y}c_{\neq})_0n_0\|_{L^2}^2\\&+\dfrac{C}{A}\|((n_{\neq}^2)_{\neq}\partial_{y}c_{\neq})_0n_0\|_{L^2}^2+\frac{C}{A}\|(n_{\neq}\nabla c_{\neq})_{0}n_{0}\|_{L^{2}}^{2}\\
				\leq&-\dfrac{C\|n_0^2\|_{L^2}^{2}}{AH_1^8}\left(\|n_0^2\|_{L^2}^{4}-H_1^{12}-1\right)+\dfrac{C}{A}\|(v_{\neq}n_{\neq})_0n_0\|_{L^2}^2\\
				&+\frac{C}{A}\|(n^2_{\neq})_0\partial_{y}c_0n_0\|_{L^{2}}^{2}+\dfrac{C}{A}\|(n_{\neq}n_0\partial_{y}c_{\neq})_0n_0\|_{L^2}^2\\&+\dfrac{C}{A}\|((n_{\neq}^2)_{\neq}\partial_{y}c_{\neq})_0n_0\|_{L^2}^2+\frac{C}{A}\|(n_{\neq}\nabla c_{\neq})_{0}n_{0}\|_{L^{2}}^{2}.    
			\end{aligned}
		\end{equation}   
		Denote
		\begin{equation*}
			\begin{aligned}  
				M(t):=&\frac{C}{A}\int_{0}^{t}\bigg(\|(v_{\neq}n_{\neq})_0n_0\|_{L^2}^2+\|(n^2_{\neq})_0\partial_{y}c_0n_0\|_{L^{2}}^{2}+\|(n_{\neq}n_0\partial_{y}c_{\neq})_0n_0\|_{L^2}^2\\
				&+\|((n_{\neq}^2)_{\neq}\partial_{y}c_{\neq})_0n_0\|_{L^2}^2+\|(n_{\neq}\nabla c_{\neq})_{0}n_{0}\|_{L^{2}}^{2} \bigg)ds,~~\forall t\geq0.
			\end{aligned}
		\end{equation*}
		It follows from assumptions (\ref{assumption_0})-(\ref{assumption_1}), (\ref{n0 infty 2}), {\textbf{Lemma \ref{ellip_0}}} and {\textbf{Lemma \ref{ellip_2}}} that
		\begin{equation*}
			\begin{aligned}
				M(t)\leq&\dfrac{C}{A}\left(\|v_{\neq}\|_{L^2L^2}^2\|n_{\neq}\|_{L^{\infty}L^{\infty}}^{2}\|n_0\|_{L^\infty L^\infty}^2+\|n_{\neq}\|_{L^2L^2}^2\|n_{\neq}\|_{L^{\infty}L^{\infty}}^{2}\|n_0\partial_{y}c_0\|_{L^\infty L^\infty}^2\right)\\
				&+\dfrac{C}{A}\left(\|\partial_{y}c_{\neq}\|_{L^2L^2}^2\|n_{\neq}n_0\|_{L^\infty L^\infty}^2\|n_{0}\|_{L^{\infty}L^{\infty}}^{2}+\|\partial_{y}c_{\neq}\|_{L^2L^2}^2\|n_{\neq}^2\|_{L^\infty L^\infty}^2\|n_{0}\|_{L^{\infty}L^{\infty}}^{2}\right)\\&+\frac{C}{A}\|n_{\neq}\|_{L^{\infty}L^{\infty}}^{2}\|\nabla c_{\neq}\|_{L^{2}L^{2}}^{2}\|n_{0}\|_{L^{\infty}L^{\infty}}^{2}\\
				\leq&\dfrac{C}{A}\left(\|\omega_{\neq}\|_{L^2L^2}^2\|n\|_{L^\infty L^\infty}^4+\|n_{\neq}\|_{L^2L^2}^2\|n\|_{L^\infty L^\infty}^4\|n_0\|_{L^\infty L^2}^2\right)\\
				&+\dfrac{C}{A}\left(\|n_{\neq}\|_{L^2L^2}^2\|n\|_{L^\infty L^\infty}^6+\|n_{\neq}\|_{L^{2}L^{2}}^{2}\|n\|_{L^{\infty}L^{\infty}}^{4} \right)\\
				\leq&\dfrac{C}{A^{\frac23}}\left(\|\omega_{\neq}\|_{X_a}^2K_{\infty}^4+\|n_{\neq}\|_{X_a}^2MK_{\infty}^5+\|n_{\neq}\|_{X_a}^2K_{\infty}^6+\|n_{\neq}\|_{X_{a}}^{2}K_{\infty}^{4}\right)\\
				\leq&\dfrac{C}{A^{\frac23}}\left(K_{\neq}^2K_{\infty}^4+MK_{\neq}^2K_{\infty}^5+K_{\neq}^2K_{\infty}^6\right)\\
				\leq&C,
			\end{aligned}
		\end{equation*}
		provided that
		\begin{equation*}
			A\geq\left(K_{\neq}^2K_{\infty}^4+MK_{\neq}^2K_{\infty}^5+K_{\neq}^2K_{\infty}^6\right)^{\frac32}.
		\end{equation*}
		Hence (\ref{energy n04}) can be rewritten as follows
		\begin{equation*}
			\dfrac{d}{dt}\left(\| n_0\|_{L^4}^4 -M(t)\right)  \leq-\dfrac{C\|n_0^2\|_{L^2}^{2}}{AH_1^8}\left(\|n_0^2\|_{L^2}^{4}-H_1^{12}-1\right),
		\end{equation*} 
		which implies that
		\begin{equation*}  
			\|n_0\|_{L^\infty L^4}\leq C\left(\|n_{\rm in,0}\|_{L^4}+H_1^{\frac32}+1\right),
		\end{equation*} 
		this gives (\ref{n0 4}).
		
		\underline{\textbf{III. Estimate of $ \omega_{0}. $}} As $ \omega_{0} $ satisfies
		\begin{equation*}
			\partial_{t}\omega_{0}-\frac{1}{A}\partial_{yy}\omega_{0}=-\frac{1}{A}\partial_{y}(v_{2,\neq}\omega_{\neq})_{0},
		\end{equation*}
		multiplying it by $2\omega_
		0$ and integrating with $ y $ over $ \mathbb{R} $, we obtain
		\begin{equation*}
			\begin{aligned}
				\frac{d}{dt}\|\omega_{0}\|_{L^{2}}^{2}+\frac{2}{A}\|\partial_{y}\omega_{0}\|_{L^{2}}^{2}=&\frac{2}{A}\int_{\mathbb{R}}(v_{2,\neq}\omega_{\neq})_{0}\partial_{y}\omega_{0}dy\\\leq&\frac{1}{A}\|\partial_{y}\omega_{0}\|_{L^{2}}^{2}+\frac{C}{A}\|v_{2,\neq}\omega_{\neq}\|_{L^{2}}^{2},
			\end{aligned}
		\end{equation*}
		which follows that
		\begin{equation}\label{energy omega 0}
			\|\omega_{0}\|_{L^{\infty}L^{2}}^{2}+\frac{1}{A}\|\partial_{y}\omega_{0}\|_{L^{2}L^{2}}^{2}\leq \|\omega_{\rm in,0}\|_{L^{2}}^{2}+\frac{C}{A}\|v_{2,\neq}\|_{L^{2}L^{\infty}}^{2}\|\omega_{\neq}\|_{L^{\infty}L^{2}}^{2}.
		\end{equation}
		Using \textbf{Lemma \ref{lem: non-zero}}, if $ A\geq K_{\neq}^{8}, $ (\ref{energy omega 0}) implies that
		\begin{equation*}
			\|\omega_{0}\|_{L^{\infty}L^{2}}^{2}+\frac{1}{A}\|\partial_{y}\omega_{0}\|_{L^{2}L^{2}}^{2}\leq C\left(\|\omega_{\rm in,0}\|_{L^{2}}^{2}+\frac{\|\omega_{\neq}\|_{X_{a}}^{4}}{A^{\frac12}} \right)\leq C\left(\|\omega_{\rm in,0}\|_{L^{2}}^{2}+1 \right).
		\end{equation*}
		This gives (\ref{omega 0}).

		\underline{\textbf{IV. Estimate of $ v_{1,0}. $}} From (\ref{v0}), we find $ v_{1,0} $ satisfies
		$$
		\partial_tv_{1,0}-\frac{1}{A}\partial_{yy}v_{1,0}=-\frac{1}{A}\partial_y(v_{2,\neq}v_{1,\neq})_0,
		$$
		and the basic energy estimates indicate that
		\begin{equation*}
			\begin{aligned}
				\frac{d}{dt}\|v_{1,0}\|_{L^{2}}^{2}+\frac{2}{A}\|\partial_{y}v_{1,0}\|_{L^{2}}^{2}=&\frac{2}{A}\int_{\mathbb{R}}(v_{2,\neq}v_{1,\neq})_{0}\partial_{y}v_{1,0}dy\\\leq&\frac{1}{A}\|\partial_{y}v_{1,0}\|_{L^{2}}^{2}+\frac{1}{A}\|(v_{2,\neq}v_{1,\neq})_{0}\|_{L^{2}}^{2}.
			\end{aligned}
		\end{equation*}
		Integrating the above inequality with $ t $, we get
		\begin{equation}\label{energy v10}
			\|v_{1,0}\|_{L^{\infty}L^{2}}^{2}+\frac{1}{A}\|\partial_{y}v_{1,0}\|_{L^{2}L^{2}}^{2}\leq \|v_{1,0}(0)\|_{L^{2}}^{2}+\frac{C}{A}\|v_{1,\neq}\|_{L^{2}L^{\infty}}^{2}\|v_{2,\neq}\|_{L^{\infty}L^{2}}^{2}.
		\end{equation}
		Using \textbf{Lemma \ref{lemma_0}} and \textbf{Lemma \ref{lem: non-zero}}, we have
		\begin{equation*}
			\|v_{2,\neq}\|_{L^{\infty}L^{2}}\leq C\|\partial_{x}v_{2,\neq}\|_{L^{\infty}L^{2}}\leq C\|\omega_{\neq}\|_{L^{\infty}L^{2}}\leq C\|\omega_{\neq}\|_{X_{a}},
		\end{equation*}
		and 
		$
		\|v_{1,\neq}\|_{L^{2}L^{\infty}}\leq CA^{\frac14}\|\omega_{\neq}\|_{X_{a}}.
		$
		If $ A\geq K_{\neq}^{8}, $ then (\ref{energy v10}) yields that
		\begin{equation}\label{energy v10 1}
			\|v_{1,0}\|_{L^{\infty}L^{2}}^{2}+\frac{1}{A}\|\partial_{y}v_{1,0}\|_{L^{2}L^{2}}^{2}\leq \|v_{1,0}(0)\|_{L^{2}}^{2}+\frac{C}{A^{\frac12}}\|\omega_{\neq}\|_{X_{a}}^{4}\leq C\left(\|v_{\rm in,0}\|_{L^{2}}^{2}+1 \right).
		\end{equation}
		Due to $ \omega_{0}=\partial_{y}v_{1,0}, $ using Gagliardo-Nirenberg inequality, there holds
		\begin{equation*}
			\|v_{1,0}\|_{L^{\infty}L^{\infty}}\leq C\|v_{1,0}\|_{L^{\infty}L^{2}}^{\frac12}\|\partial_{y}v_{1,0}\|_{L^{\infty}L^{2}}^{\frac12}=C\|v_{1,0}\|_{L^{\infty}L^{2}}^{\frac12}\|\omega_{0}\|_{L^{\infty}L^{2}}^{\frac12}.
		\end{equation*}
		Combining this with (\ref{omega 0}) and (\ref{energy v10 1}), we obtain
		\begin{equation*}
			\|v_{1,0}\|_{L^{\infty}L^{\infty}}\leq C\left(\|v_{\rm in,0}\|_{L^{2}}+\|\omega_{\rm in,0}\|_{L^{2}}+1 \right).
		\end{equation*}
		This proves (\ref{v01}).

	\end{proof}	
	\section{The estimate of $ E(t) $ and  proof of Proposition \ref{prop 1}}\label{sec 4}
	\begin{proof}[Proof of Proposition \ref{prop 1}]
		\noindent\textbf{\underline{I. Estimate $ \|\omega_{\neq}\|_{X_{a}} $ }.}
		Applying \textbf{Lemma \ref{lem:space-time}} to $(\ref{eq:non-zero mode})_{3},$ and using (\ref{assumption_0}), \textbf{Lemma \ref{lem: non-zero}} and \textbf{Lemma \ref{lem: zero mode}}, we get
		\begin{equation}\label{omega xa}
			\begin{aligned}
				\|\omega_{\neq}\|_{X_a}\leq& C\|\omega_{\rm in,\neq}\|_{L^2}+\dfrac{C}{A}\| e^{aA^{-\frac{1}{3}}t}\nabla n_{\neq}\|_{L^2L^2}+\dfrac{C}{A^{\frac{1}{2}}}\| e^{aA^{-\frac{1}{3}}t}v_{1,0}\omega_{\neq}\|_{L^2L^2}\\
				&+\frac{C}{A^{\frac12}}\left(\| e^{aA^{-\frac{1}{3}}t}v_{\neq}\omega_0\|_{L^2L^2}+\| e^{aA^{-\frac{1}{3}}t}(v_{\neq}\omega_{\neq})_{\neq}\|_{L^2L^2}\right)\\
				\leq& C\|\omega_{\rm in,\neq}\|_{L^2}+\dfrac{C}{A}\| e^{aA^{-\frac{1}{3}}t}\nabla n_{\neq}\|_{L^2L^2}+\dfrac{C}{A^{\frac{1}{2}}}\|e^{aA^{-\frac{1}{3}}t}\omega_{\neq}\|_{L^2L^2}\|v_{1,0}\|_{L^\infty L^\infty}\\
				&+\frac{C}{A^{\frac12}}\left(\| e^{aA^{-\frac{1}{3}}t}v_{\neq}\|_{L^2L^\infty}\|\omega_0\|_{L^\infty L^2}+\| e^{aA^{-\frac{1}{3}}t}v_{\neq}\|_{L^2L^\infty}\|\omega_{\neq}\|_{L^\infty L^2}\right)\\\leq&C\left(\|\omega_{\rm in,\neq}\|_{L^{2}}+\frac{1}{A^{\frac12}}\|n_{\neq}\|_{X_{a}}+\frac{H_{4}}{A^{\frac13}}\|\omega_{\neq}\|_{X_{a}}+\frac{H_{3}+\|\omega_{\neq}\|_{X_{a}}}{A^{\frac14}}\|\omega_{\neq}\|_{X_{a}} \right)\\\leq&C\left(\|\omega_{\rm in,\neq}\|_{L^{2}}+\frac{H_{3}^{2}+H_{4}^{2}+K_{\neq}^{2}+1}{A^{\frac14}} \right).
			\end{aligned}
		\end{equation}

		\noindent\textbf{~\underline{II. Estimate $ \|n_{\neq}\|_{X_{a}} $.}}
		Applying \textbf{Lemma \ref{lem:space-time}} to $(\ref{eq:non-zero mode})_{1}$, one deduces
		\begin{equation}\label{n Xa}
			\begin{aligned}
				\|n_{\neq}\|_{X_a}\leq&C\|n_{\rm in,\neq}\|_{L^2}+\frac{C}{A^{\frac12}}\left(\|e^{aA^{-\frac13}t}n_{\neq}n_{0}\partial_{y}c_{0}\|_{L^{2}L^{2}}+\|e^{aA^{-\frac13}t}(n_{\neq}^{2})_{\neq}\partial_{y}c_{0}\|_{L^{2}L^{2}} \right)\\&+\frac{C}{A^{\frac12}}\left(\|e^{aA^{-\frac13}t}n_{0}^{2}\nabla c_{\neq}\|_{L^{2}L^{2}}+\|e^{aA^{-\frac13}t}(n_{\neq}^{2})_{0}\nabla c_{\neq}\|_{L^{2}L^{2}} \right)\\&+\frac{C}{A^{\frac12}}\left(\|e^{aA^{-\frac13}t}(n_{\neq}^{2})_{\neq}\nabla c_{\neq}\|_{L^{2}L^{2}}+\|e^{aA^{-\frac13}t}n_{\neq}n_{0}\nabla c_{\neq}\|_{L^{2}L^{2}} \right)\\
				&+\frac{C}{A^{\frac12}}\left(\| e^{aA^{-\frac{1}{3}}t}v_{1,0}n_{\neq}\|_{L^2L^2}+\| e^{aA^{-\frac{1}{3}}t}v_{\neq}n_0\|_{L^2L^2}\right)\\
				&+\frac{C}{A^{\frac12}}\left(\| e^{aA^{-\frac{1}{3}}t}(v_{\neq}n_{\neq})_{\neq}\|_{L^2L^2}+\|e^{aA^{-\frac{1}{3}}t}n_0\nabla c_{\neq}\|_{L^2L^2}\right)\\
				&+\frac{C}{A^{\frac12}}\left(\|e^{aA^{-\frac{1}{3}}t}n_{\neq}\partial_{y}c_0\|_{L^2L^2}+\|e^{aA^{-\frac{1}{3}}t}(n_{\neq}\nabla c_{\neq})_{\neq}\|_{L^2L^2}\right).
	\end{aligned}	
		\end{equation}
Using (\ref{n0 infty 2}), assumptions (\ref{assumption_0})-(\ref{assumption_1}), \textbf{Lemma \ref{ellip_0}}, \textbf{Lemma \ref{ellip_2}}, \textbf{Lemma \ref{lem: non-zero}} and \textbf{Lemma \ref{lem: zero mode}}, it follows from (\ref{n Xa}) that
\begin{equation*}
			\begin{aligned}
				\|n_{\neq}\|_{X_a} \leq& C\|n_{\rm in,\neq}\|_{L^{2}}+\frac{C}{A^{\frac{1}{2}}}\|e^{aA^{-\frac13}t}n_{\neq}\|_{L^{2}L^{2}}\left(\|n_{0}\|_{L^{\infty}L^{\infty}}+\|n_{\neq}\|_{L^{\infty}L^{\infty}}\right)\|\partial_{y}c_{0}\|_{{L^{\infty}L^{\infty}}}\\
				&+\frac{C}{A^{\frac12}}\|e^{aA^{-\frac13}t}\nabla c_{\neq}\|_{L^2L^2}\left(\|n_{\neq}\|_{L^\infty L^\infty}^{2}+\|n_0\|_{L^\infty L^\infty}^{2}\right)\\&+\frac{C}{A^{\frac12}}\|v_{1,0}\|_{L^{\infty}L^{\infty}}\|e^{aA^{-\frac13}t}n_{\neq}\|_{L^{2}L^{2}}+\frac{C}{A^{\frac12}}\left(\|n_{0}\|_{L^{\infty}L^{2}}+\|n_{\neq}\|_{L^{\infty}L^{2}} \right)\|e^{aA^{-\frac13}t}v_{\neq}\|_{L^{2}L^{\infty}}\\
				&+\frac{C}{A^{\frac12}}\left(\|n_{0}\|_{L^{\infty}L^{\infty}}+\|n_{\neq}\|_{L^{\infty}L^{\infty}} \right)\|e^{aA^{-\frac13}t}\nabla c_{\neq}\|_{L^{2}L^{2}}+\frac{C}{A^{\frac12}}\|e^{aA^{-\frac13}t}n_{\neq}\|_{L^2L^2}\|\partial_{y}c_0\|_{L^\infty L^\infty}\\
				\leq&C\|n_{\rm in,\neq}\|_{L^{2}}+\dfrac{C}{A^{\frac13}}\left(M^{\frac12}K_{\infty}^{\frac32}+K_{\infty}^{2}+H_{4}+K_{\infty}+M^{\frac12}K_{\infty}^{\frac12}\right)\|n_{\neq}\|_{X_a}\\
				&+\dfrac{C}{A^{\frac14}}\left(M^{\frac12}K_{\infty}^{\frac12}+K_{\neq}\right)\|\omega_{\neq}\|_{X_a}\\
				\leq&C\left(\|n_{\rm in,\neq}\|_{L^{2}}+\dfrac{K_{\infty}^4+K_{\neq}^2+M^4+H_4^2+1}{A^{\frac14}}\right).
			\end{aligned}	
		\end{equation*}


		To sum up, we conclude that
		\begin{equation}\label{E(t)}
			E(t)\leq C\left(E_{\rm in}+\frac{K_{\infty}^{4}+K_{\neq}^{2}+M^{4}+H_{3}^{2}+H_{4}^{2}+1}{A^{\frac14}} \right),
		\end{equation}
		where
		\begin{equation*}
			\begin{aligned}
				E_{\rm in}=\|\omega_{\rm in,\neq}\|_{L^{2}}+\|n_{\rm in,\neq}\|_{L^{2}}.
			\end{aligned}
		\end{equation*}
		Let us denote $ A_{2}:=\left(K_{\infty}^{4}+K_{\neq}^{2}+M^{4}+H_{3}^{2}+H_{4}^{2}+1 \right)^{4}. $ Thus if $ A\geq A_{2}, $ (\ref{E(t)}) implies that
		\begin{equation*}
			E(t)\leq C\left(E_{\rm in}+1 \right):=K_{\neq}.
		\end{equation*}
		
		We complete the proof.
		
	\end{proof}

	\section{The $L^\infty $ estimate of the density and proof of Proposition \ref{prop 2}}\label{sec 6}
	\begin{proof}[Proof of Propsition \ref{prop 2}]
		We use mathematical induction to obtain the $ L^{\infty} $ estimate of the density. Firstly, from (\ref{assumption_0}) and (\ref{n0}), it is easy to know that
		\begin{equation}\label{n2 result}
			\|n\|_{L^{\infty}L^{2}}\leq \|n_{0}\|_{L^{\infty}L^{2}}+\|n_{\neq}\|_{L^{\infty}L^{2}}\leq H_{1}+2K_{\neq}:=D_{2}.
		\end{equation}
		
		\textbf{Step I. Estimate $ \|n\|_{L^{\infty}L^{4}}. $} It follows from \textbf{Lemma \ref{lem: zero mode}} that 
		\begin{equation}\label{no infty 4}
			\|n_{0}\|_{L^{\infty}L^{4}}\leq H_{2},
		\end{equation}
		then we only need to estimate $ \|n_{\neq}\|_{L^{\infty}L^{4}}. $
		Multiplying $(\ref{eq:non-zero mode})_1$  by $4n_{\neq}^3$ and integrating the resulting equation over $\mathbb{R}$, noting that $ \nabla\cdot v_{0}=0, $ there holds 
		\begin{equation}\label{n4}
			\begin{aligned}
				&\dfrac{d}{dt}\|n_{\neq}^2\|_{L^2}^2+\dfrac{3}{A}\|\nabla n_{\neq}^2\|_{L^2}^2\\=&\dfrac{6}{A}\int_{\mathbb{R}}n_{\neq}v_{\neq}n_0\nabla n_{\neq}^2dy+\dfrac{6}{A}\int_{\mathbb{R}}n_{\neq}(v_{\neq}n_{\neq})_{\neq}\nabla n_{\neq}^2dy+\dfrac{12}{A}\int_{\mathbb{R}}n_{\neq}^2n_0\nabla c_0\cdot\nabla n_{\neq}^2dy\\
				&+\dfrac{6}{A}\int_{\mathbb{R}}(n_{\neq}^2)_{\neq}n_{\neq}\nabla c_0\cdot\nabla n_{\neq}^2dy+\dfrac{6}{A}\int_{\mathbb{R}}n_{\neq}n_0^2\nabla c_{\neq}\cdot\nabla n_{\neq}^2dy+\dfrac{6}{A}\int_{\mathbb{R}}n_{\neq}(n_{\neq}^2)_0\nabla c_{\neq}\cdot\nabla n_{\neq}^2dy\\
				&+\dfrac{6}{A}\int_{\mathbb{R}}n_{\neq}((n_{\neq}^2)_{\neq}\nabla c_{\neq})_{\neq}\nabla n_{\neq}^2dy+\dfrac{12}{A}\int_{\mathbb{R}}n_{\neq}(n_{\neq}n_0\nabla c_{\neq})_{\neq}\nabla n_{\neq}^2dy\\
				&+\dfrac{6}{A}\int_{\mathbb{R}}n_{\neq}n_0\nabla c_{\neq}\cdot\nabla n_{\neq}^2dy+\dfrac{6}{A}\int_{\mathbb{R}}n_{\neq}^2\nabla c_0\cdot\nabla n_{\neq}^2dy+\dfrac{6}{A}\int_{\mathbb{R}}n_{\neq}(n_{\neq}\nabla c_{\neq})_{\neq}\nabla n_{\neq}^2dy\\
				\leq&\dfrac{2}{A}\|\nabla n_{\neq}^2\|_{L^2}^2+\dfrac{C}{A}\left(\|n_{\neq}v_{\neq}n_0\|_{L^2}^2+\|n_{\neq}(v_{\neq}n_{\neq})_{\neq}\|_{L^2}^2+\|n_{\neq}^2n_0\partial_{y}c_0\|_{L^2}^2\right)\\
				&+\dfrac{C}{A}\left(\|(n_{\neq}^2)_{\neq}n_{\neq}\partial_{y}c_0\|_{L^2}^2+\|n_{\neq}n_0^2\nabla c_{\neq}\|_{L^2}^2+\|n_{\neq}(n_{\neq}^2)_0\nabla c_{\neq}\|_{L^2}^2\right)\\
				&+\dfrac{C}{A}\left(\|n_{\neq}((n_{\neq}^2)_{\neq}\nabla c_{\neq})_{\neq}\|_{^2}+\|n_{\neq}(n_{\neq}n_0\nabla c_{\neq})_{\neq}\|_{L^2}^2+\|n_{\neq}n_0\nabla c_{\neq}\|_{L^2}^2\right)\\
				&+\dfrac{C}{A}\left(\|n_{\neq}^2\partial_{y}c_0\|_{L^2}^2+\|n_{\neq}(n_{\neq}\nabla c_{\neq})_{\neq}\|_{L^2}^2\right).
			\end{aligned}
		\end{equation}
		For any $ t\geq 0, $  denote	
		\begin{equation*}
			\begin{aligned}
				N(t):=&\frac{C}{A}\int_{0}^{t}\bigg(\|n_{\neq}v_{\neq}n_0\|_{L^2}^2+\|n_{\neq}(v_{\neq}n_{\neq})_{\neq}\|_{L^2}^2+\|n_{\neq}^2n_0\partial_{y}c_0\|_{L^2}^2\\
				&+\|(n_{\neq}^2)_{\neq}n_{\neq}\partial_{y}c_0\|_{L^2}^2+\|n_{\neq}n_0^2\nabla c_{\neq}\|_{L^2}^2+\|n_{\neq}(n_{\neq}^2)_0\nabla c_{\neq}\|_{L^2}^2\\
				&+\|n_{\neq}((n_{\neq}^2)_{\neq}\nabla c_{\neq})_{\neq}\|_{L^2}+\|n_{\neq}(n_{\neq}n_0\nabla c_{\neq})_{\neq}\|_{L^2}^2+\|n_{\neq}n_0\nabla c_{\neq}|_{L^2}^2\\
				&+\|n_{\neq}^2\partial_{y}c_0\|_{L^2}^2+\|n_{\neq}(n_{\neq}\nabla c_{\neq})_{\neq}\|_{L^2}^2\bigg)ds,
			\end{aligned}
		\end{equation*}
		then using the H$\ddot{\mathrm{o}}$lder's inequality, assumptions (\ref{assumption_0})-(\ref{assumption_1}), (\ref{n0 infty 2}), {\textbf{Lemma \ref{ellip_0}}} and {\textbf{Lemma \ref{ellip_2}}}, we get
		\begin{equation*}
			\begin{aligned}
				N(t)\leq&\dfrac{C}{A}\|v_{\neq}\|_{L^2L^2}^2\|n\|_{L^\infty L^\infty}^4+\dfrac{C}{A}\|n_{\neq}\|_{L^2L^2}^2\|n\|_{L^{\infty}L^{\infty}}^{4}\|\partial_{y}c_{0}\|_{L^{\infty}L^{\infty}}^{2}\\
				&+\dfrac{C}{A}\|\nabla c_{\neq}\|_{L^{2}L^{2}}^{2}\left(\|n\|_{L^{\infty}L^{\infty}}^{6}+\|n\|_{L^{\infty}L^{\infty}}^{4} \right)+\frac{C}{A}\|n_{\neq}\|_{L^{2}L^{2}}^{2}\|n\|_{L^{\infty}L^{\infty}}^{2}\|\partial_{y}c_{0}\|_{L^{\infty}L^{\infty}}^{2}\\
				\leq&\dfrac{C}{A^{\frac23}}\|\omega_{\neq}\|_{X_a}^2K_{\infty}^4+\dfrac{C}{A^{\frac23}}\|n_{\neq}\|_{X_a}^2\left(MK_{\infty}^5+K_{\infty}^6+K_{\infty}^4+MK_{\infty}^3\right)\\
				\leq&\dfrac{C}{A^{\frac23}}\left(K_{\neq}^2K_{\infty}^4+MK_{\neq}^2K_{\infty}^3+MK_{\neq}^2K_{\infty}^5+K_{\neq}^2K_{\infty}^6\right)
				\leq C,
			\end{aligned}
		\end{equation*}
		provided that
		\begin{equation*}
			A\geq\left(K_{\neq}^2K_{\infty}^4+MK_{\neq}^2K_{\infty}^3+MK_{\neq}^2K_{\infty}^5+K_{\neq}^2K_{\infty}^6\right)^{\frac32}.
		\end{equation*}
		Hence (\ref{n4}) implies that
		\begin{equation}\label{n4 4}
			\dfrac{d}{dt}\left(\|n_{\neq}^2\|_{L^2}^2-N(t)\right)\leq-\dfrac{1}{A}\|\nabla n_{\neq}^2\|_{L^2}^2\leq 0.
		\end{equation}
		By integrating both sides of (\ref{n4 4}) with $ t $, we obtain
		\begin{equation}\label{n4 result}
			\|n_{\neq}\|_{L^{\infty}L^{4}}\leq C\left(\|n_{\rm in}\|_{L^{4}}+1 \right):=H_{5}.
		\end{equation}
		Combining the estimates of (\ref{n0 4}) and (\ref{n4 result}), one deduces
		\begin{equation}\label{D4}
			\|n\|_{L^\infty L^4}\leq\|n_0\|_{L^\infty L^4}+\|n_{\neq}\|_{L^\infty L^4}\leq H_2+H_5:=D_{4}.
		\end{equation}
		
		\textbf{Step II. Estimate $ \|n\|_{L^{\infty}L^{p}}. $}
		For $ p=2^{j} $ with $ j>2 $, multiplying $ (\ref{ini2})_{1} $ by $2pn^{2p-1}$, and integrating by parts the resulting equation over $\mathbb{T}\times\mathbb{R}$, using $ \nabla\cdot v=0 $, $ (\ref{ini2})_{2} $ and $ \|c\|_{L^{2}}\leq \|n\|_{L^{2}} $, one deduces
		\begin{equation}\label{n 2p1}
			\begin{aligned}
				&\dfrac{d}{dt}\|n^p\|_{L^2}^2+\dfrac{2(2p-1)}{pA}\|\nabla n^p\|_{L^2}^2\\
				=&\frac{2p(2p-1)}{A(2p+1)}\int_{\mathbb{T}\times\mathbb{R}}\nabla c\cdot\nabla n^{2p+1}dxdy+\dfrac{2(2p-1)}{A}\int_{\mathbb{T}\times\mathbb{R}}n^p\nabla c\cdot\nabla n^pdxdy\\
				:=&S_1+S_2.
			\end{aligned}
		\end{equation}
		For $S_1$, using $ (\ref{ini2})_{2} $, the Gagliardo-Nirenberg inequality and the Young's inequality, there holds
		\begin{equation}\label{n s1}
			\begin{aligned}
				S_1=&-\frac{2p(2p-1)}{A(2p+1)}\int_{\mathbb{T}\times\mathbb{R}}n^{2p+1}\triangle cdxdy\\=&\frac{2p(2p-1)}{A(2p+1)}\int_{\mathbb{T}\times\mathbb{R}}n^{2p+2}dxdy-\frac{2p(2p-1)}{A(2p+1)}\int_{\mathbb{T}\times\mathbb{R}}n^{2p+1}cdxdy\\
				\leq&C\dfrac{2p(2p-1)}{A(2p+1)}\left(\|n^p\|_{L^2}^2\|\nabla n^p\|_{L^2}^{\frac2p}+\|n\|_{L^2}\| n^p\|_{L^2}\|\nabla n^p\|_{L^2}^{\frac{p+1}{p}}\right)\\\leq&\frac{2p-1}{2pA}\|\nabla n^{p}\|_{L^{2}}^{2}+\frac{C(2p-1)(p-1)}{Ap}\left(\frac{p}{2p+1} \right)^{\frac{p}{p-1}}\|n^{p}\|_{L^{2}}^{\frac{2p}{2p-1}}\\&+\frac{C(p-1)(2p-1)}{Ap(p+1)}\left(\frac{p(p+1)}{2p+1} \right)^{\frac{2p}{p-1}}\|n\|_{L^{2}}^{\frac{2p}{p-1}}\|n^{p}\|_{L^{2}}^{\frac{2p}{p-1}}.
			\end{aligned}
		\end{equation}
		For $S_2$, by H$\ddot{\mathrm{o}}$lder's inequality, Young's inequality and Gagliardo-Nirenberg inequality, one obtain
		\begin{equation}\label{n s2}
			\begin{aligned}
				S_2\leq&\dfrac{2(2p-1)}{A}\|\nabla n^p\|_{L^2}\|n^p\nabla c\|_{L^2}\\
				\leq&\dfrac{2p-1}{4pA}\|\nabla n^p\|_{L^2}^2+\dfrac{2p(2p-1)}{A}\|n^p\nabla c\|_{L^2}^2\\
				\leq&\dfrac{2p-1}{4pA}\|\nabla n^p\|_{L^2}^2+\dfrac{Cp(2p-1)}{A}\|n^p\|_{L^2}\|\nabla n^{p}\|_{L^{2}}\|\nabla c\|_{L^4}^2\\
				\leq&\dfrac{2p-1}{2pA}\|\nabla n^p\|_{L^2}^2+\dfrac{Cp^3(2p-1)}{A}\|n^p\|_{L^2}^2\|\nabla c\|_{L^4}^4.
			\end{aligned}
		\end{equation}
		Combining $(\ref{n s1})$ with $(\ref{n s2})$, we get by $(\ref{n 2p1})$ that
		\begin{equation}\label{n 2p}
			\begin{aligned}
				\frac{d}{dt}\|n^{p}\|_{L^{2}}^{2}\leq&-\frac{1}{A}\|\nabla n^{p}\|_{L^{2}}^{2}+\frac{Cp}{A}\|n^{p}\|_{L^{2}}^{\frac{2p}{p-1}}+\frac{Cp^{3}}{A}\|n\|_{L^{2}}^{\frac{2p}{p-1}}\|n^{p}\|_{L^{2}}^{\frac{2p}{p-1}}\\
				&+\dfrac{Cp^4}{A}\|n^p\|_{L^2}^2\|\nabla c\|_{ L^4}^4.
			\end{aligned}
		\end{equation}
		Thanks to Nash inequality
		$$
		-\|\nabla n^p\|_{L^2}^{2}\leq-C\dfrac{\|n^p\|_{L^2}^4}{\|n^p\|_{L^1}^{2}},
		$$
		 (\ref{n 2p}) yields that
		\begin{equation*}
			\begin{aligned}
				\frac{d}{dt}\|n^{p}\|_{L^{2}}^{2}\leq& -\frac{C}{A}\frac{\|n^{p}\|_{L^{2}}^{4}}{\|n^{p}\|_{L^{1}}^{2}}+\frac{Cp}{A}\|n^{p}\|_{L^{2}}^{\frac{2p}{p-1}}+\frac{Cp^{3}}{A}\|n\|_{L^{2}}^{\frac{2p}{p-1}}\|n^{p}\|_{L^{2}}^{\frac{2p}{p-1}}\\
				&+\dfrac{Cp^4}{A}\|n^p\|_{L^2}^2\|\nabla c\|_{ L^4}^4\\
				\leq&-\frac{C\|n^p\|_{L^2}^2}{A\|n^{p}\|_{L^{1}}^{2}}\bigg(\|n^{p}\|_{L^{2}}^2-p\|n^{p}\|_{L^{1}}^{2}\|n^p\|_{L^2}^{\frac{2}{p-1}}-p^{3}\|n\|_{L^{2}}^{\frac{2p}{p-1}}\|n^p\|_{L^1}^2\|n^p\|_{L^2}^{\frac{2}{p-1}}\\
				&-p^4\|n^p\|_{L^1}^2\|\nabla c\|_{L^4}^4\bigg)\\
				\leq&-\frac{C\|n^p\|_{L^2}^2}{A\|n^{p}\|_{L^{1}}^{2}}\bigg(\|n^p\|_{L^2}^2-p^{\frac{p-1}{p-2}}\|n^p\|_{L^1}^{\frac{2(p-1)}{p-2}}-p^{\frac{3(p-1)}{p-2}}\|n\|_{L^2}^{\frac{2p}{p-2}}\|n^p\|_{L^1}^{\frac{2(p-1)}{p-2}}\\
				&-p^4\|n^p\|_{L^1}^2\|\nabla c\|_{ L^4}^4\bigg),
			\end{aligned}
		\end{equation*}
		which follows that
		\begin{equation}\label{n 2p 1}
			\begin{aligned}
				\|n^{p}\|_{L^{\infty}L^{2}}^{2}\leq& C\bigg(\|n^{p}(0)\|_{L^{2}}^{2}+p^{\frac{p-1}{p-2}}\|n^{p}\|_{L^{\infty}L^{1}}^{\frac{2(p-1)}{p-2}}+p^{\frac{3(p-1)}{p-2}}\|n\|_{L^{\infty}L^{2}}^{\frac{2p}{p-2}}\|n^{p}\|_{L^{\infty}L^{1}}^{\frac{2(p-1)}{p-2}}\\
				&+p^4\|n^p\|_{L^{\infty}L^1}^2\|\nabla c\|_{L^\infty L^4}^4\bigg).
			\end{aligned}
		\end{equation}
		Using (\ref{n2 result}) and elliptic estimation, we rewrite (\ref{n 2p 1}) into
		\begin{equation}\label{ n 2p}
			\begin{aligned}
				\|n\|_{L^{\infty}L^{2p}}^{2p}\leq&C\left(\|n_{\rm in}\|_{L^{2p}}^{2p}+p^{\frac{p-1}{p-2}}\|n\|_{L^{\infty}L^{p}}^{\frac{2p(p-1)}{p-2}}+p^{\frac{3(p-1)}{p-2}}D_{2}^{\frac{2p}{p-2}}\|n\|_{L^{\infty}L^{p}}^{\frac{2p(p-1)}{p-2}}+p^{4}D_{4}^{4}\|n\|_{L^{\infty}L^{p}}^{2p} \right)\\\leq&\tilde{C}\left(1+p^{4}\|n\|_{L^{\infty}L^{p}}^{\frac{2p(p-1)}{p-2}}+\frac{p-2}{p-1}p^{\frac{4(p-1)}{p-2}}\|n\|_{L^{\infty}L^{p}}^{\frac{2p(p-1)}{p-2}} \right)\\\leq&\tilde{C}\max\left(1,p^{8}\|n\|_{L^{\infty}L^{p}}^{\frac{2p(p-1)}{p-2}} \right),
			\end{aligned}
		\end{equation}
		where $ \tilde{C} $ is a positive constant depending on $ \|n_{\rm in}\|_{L^{\infty}}, D_{2} $ and $ D_{4}. $
		
		\textbf{Step III. Estimate $ \|n\|_{L^{\infty}L^{\infty}}. $}
		Recall $ p=2^{j} $ for $ j\geq 2, $ then using iterative estimation, it follows from (\ref{ n 2p}) that
			\begin{equation}\label{iterative}
				\begin{aligned}
					&\sup_{t\geq 0}\int_{\mathbb{T}\times\mathbb{R}}|n|^{2^{j+1}}dxdy\\\leq&\tilde{C}\max\left[1,2^{8j}\left(\sup_{t\geq 0}\int_{\mathbb{T}\times\mathbb{R}}|n|^{2^{j}}dxdy \right)^{2\left(\frac{2^{j}-1}{2^{j}-2} \right)} \right]\\\leq&\tilde{C}\max\left\{1,2^{8j}\left[\tilde{C}\max\left(1,2^{8(j-1)}\left(\sup_{t\geq 0}\int_{\mathbb{T}\times\mathbb{R}}|n|^{2^{j-1}}dxdy \right)^{2\left(\frac{2^{j-1}-1}{2^{j-1}-2} \right)} \right) \right]^{2\left(\frac{2^{j}-1}{2^{j}-2} \right)} \right\}\\
					\leq&\tilde{C}+\tilde{C}^{1+{\frac{2^j-1}{2^{j-1}-1}}}2^{8j}\\
					&+\cdots+\tilde{C}^{1+\sum_{k=1}^{j-2}2^{k}\prod_{l=1}^{k}\frac{2^{j-l+1}-1}{2^{j-l+1}-2}}2^{8j+8\sum_{k=1}^{j-3}2^{k}(j-k)\prod_{l=1}^{k}\frac{2^{j-l+1}-1}{2^{j-l+1}-2}}\\ &+\tilde{C}^{1+\sum_{k=1}^{j-2}2^{k}\prod_{l=1}^{k}\frac{2^{j-l+1}-1}{2^{j-l+1}-2}}2^{8j+8\sum_{k=1}^{j-2}2^{k}(j-k)\prod_{l=1}^{k}\frac{2^{j-l+1}-1}{2^{j-l+1}-2}}D_{4}^{4\sum_{k=1}^{j-1}2^{k}\prod_{l=1}^{k}\frac{2^{j-l+1}-1}{2^{j-l+1}-2}}\\				\leq&\tilde{C}^{1+\sum_{k=1}^{j-2}2^{k}\prod_{l=1}^{k}\frac{2^{j-l+1}-1}{2^{j-l+1}-2}}2^{8j+8\sum_{k=1}^{j-2}2^{k}(j-k)\prod_{l=1}^{k}\frac{2^{j-l+1}-1}{2^{j-l+1}-2}}D_{4}^{4\sum_{k=1}^{j-1}2^{k}\prod_{l=1}^{k}\frac{2^{j-l+1}-1}{2^{j-l+1}-2}}\\
					&+(j-1)\tilde{C}^{1+\sum_{k=1}^{j-2}2^{k}\prod_{l=1}^{k}\frac{2^{j-l+1}-1}{2^{j-l+1}-2}}2^{8j+8\sum_{k=1}^{j-3}2^{k}(j-k)\prod_{l=1}^{k}\frac{2^{j-l+1}-1}{2^{j-l+1}-2}}.
				\end{aligned}
			\end{equation}
		Due to 
		\begin{equation*}
			\begin{aligned}
				&\sum_{k=1}^{j-1}2^{k}\prod_{l=1}^{k}\frac{2^{j-l+1}-1}{2^{j-l+1}-2}\\=&\sum_{k=1}^{j-1}2^{k}\left(\frac{2^{j}-1}{2^{j}-2} \right)\times\left(\frac{2^{j-1}-1}{2^{j-1}-2} \right)\times\cdots\times\left(\frac{2^{j-k+2}-1}{2^{j-k+2}-2} \right)\times\left(\frac{2^{j-k+1}-1}{2^{j-k+1}-2} \right)\\=&\sum_{k=1}^{j-1}2^{k}\frac{2^j-1}{2^k\left(2^{j-k}-1\right)}=\sum_{k=1}^{j-1}\frac{2^j-1}{2^{j-k}-1},
			\end{aligned}
		\end{equation*}	
		and
		\begin{equation*}
			\begin{aligned}
				&8j+8\sum_{k=1}^{j-2}2^{k}(j-k)\prod_{l=1}^{k}\frac{2^{j-l+1}-1}{2^{j-l+1}-2}
=8j+8\sum_{k=1}^{j-2}(j-k)\frac{2^{j+1}-2}{2^{j+1-k}-2},
			\end{aligned}
		\end{equation*}
		there hold
		\begin{equation*}
			\begin{aligned}
			0\leq \lim_{j\to\infty}\frac{1}{2^{j+1}}\left(\sum_{k=1}^{j-1}2^{k}\prod_{l=1}^{k}\frac{2^{j-l+1}-1}{2^{j-l+1}-2} \right)=&\lim_{j\to\infty}\left[\frac12\left(1-\frac{1}{2^{j}}\right)\sum_{k=1}^{j-1}\frac{1}{2^{j-k}-1}\right] \\\leq&\frac12\sum_{n=0}^{\infty}\frac{1}{2^{n}}=1,
			\end{aligned}
		\end{equation*}
		and
		\begin{equation*}
			\begin{aligned}
			0\leq&\lim_{j\to\infty}\frac{1}{2^{j+1}}\left[8j+8\sum_{k=1}^{j-2}2^{k}(j-k)\prod_{l=1}^{k}\frac{2^{j-l+1}-1}{2^{j-l+1}-2} \right]\\=&8\lim_{j\to\infty}\frac{j}{2^{j+1}}+8\lim_{j\to\infty}\left(\frac{2^{j+1}-2}{2^{j+1}}\sum_{k=1}^{j-2}\frac{j-k}{2^{j+1-k}-2} \right)\leq 4\sum_{n=0}^{\infty}\frac{n+1}{2^{n}}=12,
			\end{aligned}
		\end{equation*}
		which along  with (\ref{iterative}) implies that
		\begin{equation*}
			\begin{aligned}
				\|n\|_{L^{\infty}L^{\infty}}\leq \hat{C}:=K_{\infty},
			\end{aligned}
		\end{equation*}
		where $\hat{C}$ is a positive constant depending on $\|n_{\rm in}\|_{L^{1}\cap L^{\infty}}$.

		To sum up, we complete the proof.
		
	\end{proof}
	
	\section{Proof of Theorem \ref{result 2}}\label{theorem}
	\begin{proof}[Proof of Theorem \ref{result 2}]
		Recall (\ref{c*}) and (\ref{energy n0 1}) in the proof of {\textbf{Lemma \ref{lem: zero mode}}}, where we use the interpolation inequality
		\begin{equation*}
			\|n_{0}\|_{L^{4}}\leq C_{*}\|n_{0}\|_{L^{1}}^{\frac12}\|\partial_{y}n_{0}\|_{L^{2}}^{\frac12},
		\end{equation*}
		and  $C_{*}^{4}M^{2}<3. $  From \textbf{Remark \ref{interpolation}},  notice that $ C_{*}=\left(\frac{4\pi^{2}}{9}\right)^{-\frac14}, $ which follows that $ M<\frac{2\pi}{\sqrt{3}}. $
		Then combining it with {\textbf{Theorem \ref{result}}}, the proof is complete.
	\end{proof}
	
	\section*{Acknowledgement}
	The authors would like to thank Professor Wei Wang and Dr. Shikun Cui for some helpful communications. W. Wang was supported by National Key R\&D Program of China (No. 2023YFA1009200) and NSFC under grant 12071054 and 12471219.
	
	\section*{Declaration of competing interest}
	The authors declare that they have no known competing financial interests
	or personal relationships that could have appeared to
	influence the work reported in this paper.
	\section*{Data availability}
	No data was used in this paper.


\begin{thebibliography}{99}	
		\bibitem{Bedro2}Bedrossian J. and He S. (2017). Suppression of blow-up in Patlak--Keller--Segel via shear flows.
		SIAM Journal on Mathematical Analysis, 49(6), 4722-4766.
		\bibitem{Calvez1}Calvez V. and Corrias L. (2008).
		The parabolic-parabolic Keller-Segel model in $\mathbb{R}^2$. Communications in Mathematical Sciences, 6(2), 417-447.
		\bibitem{1993} Carlen A. and Loss M. (1993). Sharp constant in Nash's inequality. International Mathematics Research Notices, 7, 213-215.
		
		


\bibitem{CS2} Cie\'{s}lak, Tomasz; Stinner, Christian;
New critical exponents in a fully parabolic quasilinear Keller-Segel system and applications to volume filling models. (English summary)
J. Differential Equations 258 (2015), no. 6, 2080-2113.



		\bibitem{cui1} Cui S. and Wang W. (2023). Suppression of blow-up in multi-species Patlak-Keller-Segel-Navier-Stokes system via the Poiseuille flow in a finite channel. arXiv:2311.18519.
		
		
		\bibitem{cui3} Cui S., Wang L. and Wang W. (2024). Suppression of blow-up in
		Patlak-Keller-Segel system coupled with linearized Navier-Stokes equations via the 3D Couette flow. arXiv:2401.15982.
		\bibitem{wangweike1}Deng S., Shi B. and Wang W. (2024).
		Suppression of blow-up in 3-D Keller-Segel model via Couette flow in whole space. arXiv:2311.18590.
		\bibitem{he0}He S. (2018). Suppression of blow-up in parabolic-parabolic Patlak-Keller-Segel via strictly monotone shear flows. Nonlinearity, 31(8), 3651.
		\bibitem{HW}Horstmanna D., Winkler M. (2005). Boundedness vs. blow-up in a chemotaxis system. J. Differential Equations, 215, 52-107.
		\bibitem{Hu2023} Hu Z. (2023). Suppression of Chemotactic Singularity via Viscous Flow with Large Buoyancy. arXiv:2311.10003.
		\bibitem{Hu0} Hu Z., Kiselev A. and Yao Y. (2023) Suppression of chemotactic singularity by buoyancy. arXiv:2305.01036.
		\bibitem{Keller1}Keller E.F. and Segel L.A. (1970). Initiation of slime mold aggregation viewed as an instability.
		Journal of Theoretical Biology, 26(3), 399.
		\bibitem{Kiselev1}Kiselev A. and Xu X. (2016).
		Suppression of chemotactic explosion by mixing. Archive for Rational Mechanics and Analysis, 222, 1077-1112.
		\bibitem{Li0}Li H., Xiang Z. and Xu X. (2023). Suppression of blow-up in Patlak-Keller-Segel-Navier-Stokes system via the Poiseuille flow, arXiv:2312.01069.
		\bibitem{liu2017}Liu J. and Wang J. (2017). One the best constant for Gagliardo-Nirenberg interpolation inequalities, arXiv:1712.10208v1.
		%
		%
		
		\bibitem{nagai1995} Nagai T. (1995). Blow-up of radially symmetric solutions to a chemotaxis system. Adv. Math. Sci. Appl. 5, no. 2, 581–601.  
		
		
		
		\bibitem{Na} Nagai T. (2000). Behavior of solutions to a parabolic-elliptic system modelling chemotaxis. J. Korean Math. Soc. 37, 721-732.
		
		
		\bibitem{nagy 1941} Nagy B. V. Sz. (1941). ${\rm\ddot{U}}$cber integralungleichungen zwischen einer funktion aun ihrer ableitung. Acta Univ. Szeged. Sect. Sci. Math. 10, 64-74.
		
		\bibitem{painter}Painter K. and Hillen T. (2002). Volume-filling and quorum-sensing in models for chemosensitive movement. Canadian Applied Mathematics Quarterly. 10 (4), 501-543.
		\bibitem{Patlak1}Patlak C.S. (1953). Random walk with persistence and external bias.
		The bulletin of mathematical biophysics, 15, 311-338.
		\bibitem{Schweyer1}Schweyer R. (2014). Stable blow-up dynamic for the parabolic-parabolic Patlak-Keller-Segel model.
		arxiv:1403.4975.
		\bibitem{SW}Shi B. and Wang W. (2024). Enhanced dissipation and blow-up suppression for the three dimensional Keller-Segel equation with the plane Couette-Poiseuille flow. J. Differential Equations, 403, 368-405. 
		
		
		\bibitem{sw2019} Souplet P. and Winkler M. (2019). Blow-up profiles for the parabolic-elliptic Keller-Segel system in dimensions $n\geq3$. Commun. Math. Phys. 367, 665–681.
		
		\bibitem{SM} Stinner C. and Winkler M. (2024). A critical exponent in a quasilinear Keller–Segel system with arbitrarily fast decaying diffusivities accounting for volume-filling effects. J. Evol. Equ. 24-26.
		\bibitem{wang2024} Wang W. and Liu Z. (2024). Globally bounded solutions in a 2D chemotaxis-Navier-Stokes system with general sensitivity and nonlinear production. Z. Angew. Math. Phy. 75:74.
		\bibitem{wei11}Wei D. (2018). Global well-posedness and blow-up for the 2-D Patlak-Keller-Segel equation.
		Journal of Functional Analysis, 274(2), 388-401.
		\bibitem{winkler0}Winkler M. (2009). Does a ‘volume-filling effect’ always prevent chemotactic collapse? Math. Methods Appl. Sci. 33, 12-24.
		\bibitem{winkler1}Winkler M. (2013). Finite-time blow-up in the higher-dimensional parabolic-parabolic Keller-Segel system. Journal de Mathematiques Pures et Appliquees, 100(5), 748-767.
		\bibitem{winkler2}Winkler M. and Djie K C. (2010). Boundedness and finite-time collapse in a chemotaxis system with volume-filling effect. Nonlinear Anal. 72, 1044-1064.
		\bibitem{zeng}Zeng L., Zhang Z.  and Zi R. (2021). Suppression of blow-up in
		Patlak-Keller-Segel-Navier-Stokes system via the couette flow. Journal of Functional Analysis, 280(10), 108967.
		
		
		

    
		
	\end{thebibliography}
\end{document}